\documentclass[reqno]{amsart}
\usepackage{amsmath,latexsym,amssymb}
\usepackage[mathscr]{eucal}
\usepackage{enumitem}

\theoremstyle{plain}
\newtheorem{theorem}{Theorem}[section]
\newtheorem{lemma}[theorem]{Lemma}
\newtheorem{corollary}[theorem]{Corollary}
\newtheorem{example}[theorem]{Example}

\newtheorem{claim}{Claim}[theorem]

\theoremstyle{definition}
\newtheorem{definition}[theorem]{Definition}
\newtheorem{remark}[theorem]{Remark}
\newtheorem{note}[theorem]{Note}
\newtheorem{assumptions}[theorem]{Assumptions}
\newtheorem{chatexample}[theorem]{Example}

\newcommand{\Case}[1]{\smallskip\noindent{\emph{#1}}}

\newenvironment{proofofclaim}[1]%
 {\begin{proof}[Proof of Claim~#1]}%
 {\end{proof}}

\usepackage{tikz,pgf}
\usetikzlibrary{arrows,
calc,shapes,backgrounds,fit}

\tikzset{%
 element/.style={draw,circle,fill=white,inner sep=1.5pt},
 shaded element/.style={draw,circle,fill=black!30,inner sep=1.5pt},
 potato/.style={rounded rectangle,draw,densely dotted,inner sep=1.5pt},
 operation/.style={semithick,densely dotted,shorten >=3pt,shorten <=3pt,>=latex},
 loopy operation/.style={semithick,densely dotted,shorten >=3pt,shorten <=3pt,>=latex, min distance=18pt},
 graph/.style={thin,shorten >=3pt,shorten <=3pt},
 loopy graph/.style={thin,shorten >=3pt,shorten <=3pt, min distance=18pt},
 move up/.style= {transform canvas={yshift=2pt}},
 move down/.style={transform canvas={yshift=-2pt}},
 move far down/.style={transform canvas={yshift=-4pt}},
 move left/.style= {transform canvas={xshift=-2pt}},
 move right/.style={transform canvas={xshift=2pt}},
 auto}

\DeclareMathOperator{\Image}{Im}
\DeclareMathOperator{\Sub}{Sub}
\DeclareMathOperator{\Var}{Var}
\DeclareMathOperator{\Up}{Up}
\DeclareMathOperator{\Con}{Con}

\DeclareMathOperator{\Equiv}{Equiv}
\DeclareMathOperator{\Max}{Max}
\DeclareMathOperator{\Min}{Min}
\DeclareMathOperator{\Cyc}{Cyc}
\DeclareMathOperator{\DM}{DM}

\newcommand{\fin}{_{\mathrm{fin}}}
\newcommand{\id}{\mathrm{id}}

\newcommand{\dotcup}{\mathbin{\dot\cup}}
\newcommand{\dnarrow}{{\downarrow}}
\newcommand{\uarrow}{{\uparrow}}

\newcommand{\restrictedto}[1]{{\upharpoonright_{#1}}}

\newcommand{\A}{{\mathbf A}}
\newcommand{\B}{{\mathbf B}}
\newcommand{\C}{{\mathbf C}}
\newcommand{\D}{{\mathbf D}}
\newcommand{\F}{{\mathbf F}}
\newcommand{\G}{{\mathbf G}}
\newcommand{\Hb}{{\mathbf H}}
\newcommand{\K}{{\mathbf K}}
\newcommand{\Lb}{{\mathbf L}}
\newcommand{\M}{{\mathbf M}}
\newcommand{\N}{{\mathbf N}}
\newcommand{\Pb}{{\mathbf P}}
\newcommand{\Q}{{\mathbf Q}}
\newcommand{\Sb}{{\mathbf S}}
\newcommand{\U}{{\mathbf U}}
\newcommand{\V}{{\mathbf V}}
\newcommand{\1}{{\mathbf 1}}
\newcommand{\2}{{\mathbf 2}}

\newcommand{\cat}[1]{{\mathscr #1}}
\newcommand{\Lhom}{{\mathcal L_{\mathrm{hom}}}}
\newcommand{\upset}{{\mathcal U}}
\newcommand{\downsets}[1]{\mathcal O(#1)}
\newcommand{\properdownsets}[1]{\mathcal O^+(#1)}
\newcommand{\notop}{\setminus\{\top\}}

\newcommand{\pmods}[1]{\mkern8mu({\operatorname{mod}}\mkern 6mu#1)}

\renewcommand{\phi}{\varphi}
\renewcommand{\ge}{\geqslant}
\renewcommand{\le}{\leqslant}
\renewcommand{\emptyset}{\varnothing}

\begin{document}

\title[The homomorphism lattice induced by a finite algebra]%
  {The homomorphism lattice\\ induced by a finite algebra}

\author{Brian A. Davey}
  \address[Brian A. Davey]{Department of Mathematics and Statistics\\La Trobe University\\
    Victoria 3086, Australia}
  \email{B.Davey@latrobe.edu.au}

\author{Charles T. Gray}
  \address[Charles T. Gray]{Department of Mathematics and Statistics\\La Trobe University\\
    Victoria 3086, Australia}
  \email{C.Gray@latrobe.edu.au}

\author{Jane G. Pitkethly}
  \address[Jane G. Pitkethly]{Department of Mathematics and Statistics\\La Trobe University\\
    Victoria 3086, Australia}
  \email{J.Pitkethly@latrobe.edu.au}



\begin{abstract}
Each finite algebra $\A$ induces a lattice~$\Lb_\A$ via
the quasi-order~$\to$ on the finite members of the variety generated by~$\A$,
where $\B \to \C$ if there exists a homomorphism from $\B$ to~$\C$.
In this paper, we introduce the question: `Which lattices arise as the
homomorphism lattice $\Lb_\A$ induced by a finite algebra $\A$?'
Our main result is that each finite distributive lattice arises as~$\Lb_\Q$,
for some quasi-primal algebra~$\Q$. We also obtain representations of some
other classes of lattices as homomorphism lattices, including
all finite partition lattices, all finite subspace lattices and all lattices
of the form $\Lb\oplus \1$, where $\Lb$ is an interval in the subgroup
lattice of a finite group.
\end{abstract}

\maketitle


For any category~$\cat C$, there is a natural quasi-order $\to$ on the objects of~$\cat C$,
given by $\A \to \B$ if there exists a morphism from $\A$ to~$\B$.
The associated equivalence relation on~$\cat C$ is given by
\[
\A \equiv \B \iff \A \to \B \text{ and } \B \to \A.
\]
Thus $\to$ induces an order on the class $\cat C/{\equiv}$.
One way to ensure that $\cat C/{\equiv}$ is a set is to take the category $\cat C$
to be a class of finite structures with all homomorphisms between them.
In this case, we can define the ordered set
\[
 \Pb_{\cat C} := \langle \cat C/{\equiv}; \to \rangle,
\]
which we refer to as the \emph{homomorphism order} on~$\cat C$.
If $\cat C$ has pairwise products and coproducts, then $\Pb_{\cat C}$  is a lattice:
the meet and join of $\A/{\equiv}$ and $\B/{\equiv}$ are $(\A\times \B)/{\equiv}$
and~$(\A\sqcup \B)/{\equiv}$, respectively; see Figure~\ref{fig: graphs}.

\begin{figure}[t]
\begin{tikzpicture}
\node (prod) at (0,0) {$\A \times \B$};
\node (G) at (-1,1) {$\A$};
\node (H) at (1,1) {$\B$};
\node (coprod) at (0,2) {$\A \sqcup \B$};
\draw[->] (prod) to (G);
\draw[->] (prod) to (H);
\draw[->] (G) to (coprod);
\draw[->] (H) to (coprod);
\end{tikzpicture}
\caption{Meet and join in $\Pb_{\cat C} $ when $\cat C$ has pairwise products and coproducts.}\label{fig: graphs}
\end{figure}

The homomorphism order has been studied extensively for the category $\cat G$
of finite directed graphs; see Hell and Ne\v{s}et\v{r}il~\cite{HN}.
The ordered set $\Pb_{\cat G}$,
which forms a bounded distributive lattice, is very complicated:
every countable ordered set embeds into~$\Pb_{\cat G}$~\cite{H69,PT,HuN}.
More generally, the homomorphism order has been studied for various categories
of finite relational structures~\cite{NT,NPT,FNT,KL}.

In this paper, we introduce the study of the homomorphism order for categories
of the form~$\cat V\fin$, consisting of the finite members of a variety $\cat V$ of algebras.
To ensure that the homomorphism order forms a lattice, we shall restrict our
attention to locally finite varieties and, more particularly, to finitely generated varieties.

Given a finite algebra $\A$, we can define the lattice
\[
 \Lb_\A := \langle \Var(\A)\fin/{\equiv}; \to \rangle,
\]
which we refer to as the \emph{homomorphism lattice} induced by~$\A$.
We shall see that such a lattice $\Lb_\A$ may be just as complicated as the
homomorphism order $\Pb_{\cat G}$ for finite directed graphs.
For example, there is a five-element unary algebra $\U$ such that $\Pb_{\cat G}$
order-embeds into~$\Lb_{\U}$; see Example~\ref{ex: inf}.

We are interested in the question:
\begin{quote}
Which lattices arise as $\Lb_\A$, for some finite algebra~$\A$?
\end{quote}
Our main result (proved over Sections~\ref{sec: Q to L}--\ref{sec: L to Q})
is that each finite distributive lattice arises as the homomorphism lattice~$\Lb_\Q$,
for some quasi-primal algebra~$\Q$. In the proof, we use Behncke and Leptin's
construction~\cite{BL} of the \emph{covering forest} of a finite ordered set,
which is analogous to the universal covering tree from graph theory.

In Section~\ref{sec: extra}, we obtain representations for some other classes of finite lattices.
We consider finite algebras $\A$ such that each element is the value of a nullary term function,
and prove a simple result (Lemma~\ref{lem: new}) characterising when such an algebra satisfies
$\Lb_\A \cong \Con(\A)$. This allows us to represent a range of finite lattices as the
homomorphism lattice induced by a finite algebra:
\begin{itemize}
\item
every finite partition lattice (Example~\ref{ex: retracts});
\item
the lattice of subspaces of a finite vector space (Example~\ref{ex: retracts});
\item
the lattice $[\Hb, \G] \oplus \mathbf 1$, where $[\Hb, \G]$ denotes an
interval in the subgroup lattice of a finite group $\G$ (Example~\ref{ex: subgroups});
\item
the five-element non-modular lattice $\N_5$ (Example~\ref{ex: pentagon}).
\end{itemize}
We use this representation for finite partition lattices to see that the only universal
first-order sentences true in all homomorphism lattices are those true in all lattices.

There are many unanswered questions concerning the homomorphism lattices
induced by finite algebras. For example:
\begin{itemize}
\item
Does every countable bounded lattice arise as the homomorphism lattice~$\Lb_\A$
induced by a finite algebra~$\A$? In particular, does every finite lattice
arise in this way?
\item
For which finite algebras $\A$ is the lattice $\Lb_\A$ finite? Is this decidable?
\end{itemize}

\section{The homomorphism order}\label{sec: hom order}

In this introductory section, we motivate the \emph{homomorphism lattice} induced by
a finite algebra and give some examples.

\subsection*{The homomorphism order on finite directed graphs}

The motivation for this paper came from Hell and Ne\v{s}et\v{r}il's
text \emph{Graphs and Homomorphisms}~\cite{HN}.

Recall that $\cat G$ denotes the category of finite directed graphs.
Since $\cat G$ has pairwise products, given by direct product,
and coproducts, given by disjoint union, the ordered set
$\Pb_{\cat G} = \langle \cat G/{\equiv}; \to \rangle$ forms a lattice.
Since product distributes over disjoint union, the lattice $\Pb_{\cat G}$ is distributive.
In fact, the lattice $\Pb_{\cat G}$ is relatively pseudocomplemented, via the
exponential construction (see~\cite[Section 2.4]{HN}).

Both the lattice $\Pb_{\cat G}$ and its sublattice $\Pb_{\cat S}$ have been studied extensively,
where $\cat S$ is the category of finite symmetric directed graphs (i.e., finite graphs).
For example, it is known that every countable ordered set embeds into $\Pb_{\cat S}$~\cite{PT,HuN},
and that $\Pb_{\cat S}$ is dense above the complete graph~$\K_2$~\cite{Welzl}
(see~\cite[Section 3.7]{HN}).

\subsection*{The homomorphism order on categories of algebras}

Within many natural categories of finite algebras, all the algebras are
homomorphically equivalent, and so the homomorphism order is trivial: groups, semigroups,
rings and lattices, for example. However, there are also many natural categories of
finite algebras for which the homomorphism order is extremely complicated.

\begin{chatexample}\label{ex: L01}
Consider the category $\cat L_{01}$ of finite bounded lattices. A simple observation is that
there is an infinite ascending chain $\M_3 \to \M_4 \to \M_5 \to \dotsb$
in the homomorphism order $\Pb_{\cat L_{01}}$, where $\M_n$ is the bounded lattice of height~2
with $n$~atoms. In fact, we can say much more.

A variety $\cat V$ of algebras is \emph{finite-to-finite universal} if
the category of directed graphs has a finiteness-preserving full embedding into~$\cat V$.
(This is equivalent to requiring that every variety of algebras has a finiteness-preserving
full embedding into~$\cat V$; see \cite{Pu,HP,PT}). Since the variety of bounded
lattices is finite-to-finite universal (Adams and Sichler~\cite{AS}), it follows that there
is an order-embedding of $\Pb_{\cat G}$ into~$\Pb_{\cat L_{01}}$, and therefore every
countable ordered set embeds into~$\Pb_{\cat L_{01}}$.
\end{chatexample}

\subsection*{The homomorphism lattice induced by a finite algebra}

In general, the coproduct of two finite algebras in a variety does not have to be finite.
Consequently, it is not clear whether the homomorphism order $\Pb_{\cat L_{01}}$ from
Example~\ref{ex: L01} is a lattice. We can avoid this problem if we restrict our attention
to locally finite varieties.

\begin{lemma}
Let $\cat V$ be a locally finite variety. Then the homomorphism order
$\Pb_{\cat V\fin} =  \langle \cat V\fin/{\equiv}; \to \rangle$ is a
countable bounded lattice.
\end{lemma}

\begin{proof}
Let $\A, \B \in \cat V\fin$. To see that $\Pb_{\cat V\fin}$ is a lattice,
it suffices to observe that the product $\A \times \B$ is finite and
the coproduct $\A \sqcup \B$ in $\cat V$ is finite (since $\cat V$
is locally finite); see Figure~\ref{fig: graphs}.
The top element of $\Pb_{\cat V\fin}$ contains all the trivial algebras in~$\cat V$
(and consists of all the finite algebras in $\cat V$ with a trivial subalgebra).
The bottom element of $\Pb_{\cat V\fin}$ contains all the finitely generated
free algebras in~$\cat V$.
Note that $\Pb_{\cat V\fin}$ is countable as every finite algebra in $\cat V$
is a homomorphic image of a finitely generated free algebra.
\end{proof}

In this paper, we focus on finitely generated varieties.
Given a finite algebra $\A$, we can define the \emph{homomorphism lattice} induced by~$\A$ to be
\[
 \Lb_\A = \langle \Var(\A)\fin/{\equiv}; \to \rangle.
\]

If the variety $\Var(\A)$ is finite-to-finite universal, then the lattice $\Lb_\A$
is just as complicated as the homomorphism order $\Pb_{\cat G}$ for finite directed graphs,
as $\Pb_{\cat G}$ order-embeds into~$\Lb_\A$. Examples of finite algebras that generate a
finite-to-finite universal variety include:
\begin{itemize}
\item the bounded lattice $\M_3$~\cite{GKS};
\item the algebra $\A = \langle A; \vee,\wedge, 0, 1, a_1, a_2\rangle$, where
  $\langle A; \vee,\wedge, 0, 1\rangle$ is the bounded distributive lattice $\1\oplus \2^2\oplus \1$
  freely generated by $\{a_1, a_2\}$~\cite{AKS}.
\end{itemize}
Even a small unary algebra can generate a finite-to-finite universal variety, as in the following example.

\begin{example}\label{ex: inf}
Let $\U = \langle \{ 0, 1, 2, u, v \}; f_0, f_1 \rangle$ be the five-element unary algebra
shown in Figure~\ref{fig: u}. Then $\Pb_{\cat G}$ order-embeds into~$\Lb_\U$, and therefore
every countable ordered set embeds into~$\Lb_\U$.

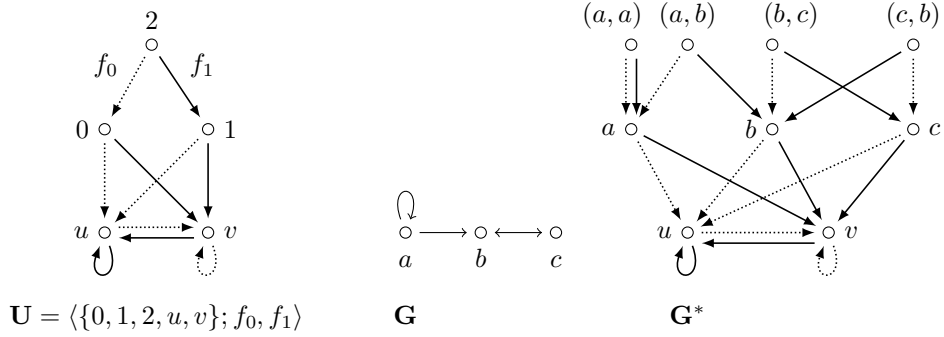
\begin{figure}
\begin{tikzpicture}
\begin{scope}[node distance=1.375cm]
  \node[element,label=left: {$0$}] (0) {};
  \node[element,right of=0,label=right: {$1$}] (1) {};
  \node[element,below of=0,label=left: {$u$}] (u) {};
  \node[element,right of=u,label=right: {$v$}] (v) {};
  \node[element,above of=0,xshift=0.625cm,yshift=-0.25cm,label=above: {$2$}] (e) {};
  \coordinate (uv) at ($0.5*(u)+0.5*(v)$);
  \node[below of=uv,yshift=0.25cm] {$\U = \langle \{ 0, 1, 2, u, v \}; f_0, f_1 \rangle$};
  \path[->,operation,solid] (e) edge node[above right] {$f_1$} (1);
  \path[->,operation,solid] (0) edge (v);
  \path[->,operation,solid] (1) edge (v);
  \path[->,operation,solid,move down] (v) edge (u);
  \path[->,loopy operation,solid] (u) edge [out=290,in=250] (u);
  \path[->,operation] (e) edge node[above left] {$f_0$} (0);
  \path[->,operation] (0) edge (u);
  \path[->,operation] (1) edge (u);
  \path[->,operation,move up] (u) edge (v);
  \path[->,loopy operation] (v) edge [out=290,in=250] (v);
\end{scope}
\begin{scope}[xshift=4cm,yshift=-1.375cm,node distance=1cm]
  \node[element,label=below: {$a\vphantom{b}$}] (a) {};
  \node[element,right of=a,label=below: {$b$}] (b) {};
  \node[element,right of=b,label=below: {$c\vphantom{b}$}] (c) {};
  \draw[->,graph] (a) to (b);
  \draw[<->,graph] (b) to (c);
  \path[->,loopy graph] (a) edge [out=110,in=70] (a);
  \node[below of=a,yshift=-0.125cm] {$\G\vphantom{\rangle}$};
\end{scope}
\begin{scope}[xshift=7cm,node distance=1.375cm]
  \node[element,label=left: {$a$}] (a) {};
  \node[element,right of=a,xshift=0.5cm,label=left: {$b$}] (b) {};
  \node[element,right of=b,xshift=0.5cm,label=right: {$c$}] (c) {};
  \node[element,below of=a,xshift=0.75cm,label=left: {$u$}] (u) {};
  \node[element,right of=u,xshift=0.5cm,label=right: {$v$}] (v) {};
  \node[element,above of=a,xshift=0.75cm,yshift=-0.25cm,label=above: {$(a,b)$}] (ab) {};
  \node[element,above of=a,yshift=-0.25cm,label=above: {$(a,a)$\hspace*{0.5cm}}] (aa) {};
  \node[element,above of=b,yshift=-0.25cm,label=above: {\hspace*{0.5cm}$(b,c)$}] (bc) {};
  \node[element,above of=c,yshift=-0.25cm,label=above: {$(c,b)$}] (cb) {};
  \node[below of=u,yshift=0.25cm] {$\G^*\vphantom{\rangle}$};
  \path[->,operation,solid,move right] (aa) edge (a);
  \path[->,operation,solid] (ab) edge (b);
  \path[->,operation,solid] (bc) edge (c);
  \path[->,operation,solid] (cb) edge (b);
  \path[->,operation,solid] (a) edge (v);
  \path[->,operation,solid] (b) edge (v);
  \path[->,operation,solid] (c) edge (v);
  \path[->,operation,solid,move far down] (v) edge (u);
  \path[->,loopy operation,solid] (u) edge [out=290,in=250] (u);
  \path[->,operation,move left] (aa) edge (a);
  \path[->,operation] (ab) edge (a);
  \path[->,operation] (bc) edge (b);
  \path[->,operation] (cb) edge (c);
  \path[->,operation] (a) edge (u);
  \path[->,operation] (b) edge (u);
  \path[->,operation] (c) edge (u);
  \path[->,operation] (u) edge (v);
  \path[->,loopy operation] (v) edge [out=290,in=250] (v);
\end{scope}
\end{tikzpicture}
\caption{A finite unary algebra $\U$ with infinite lattice $\Lb_\U$.}\label{fig: u}
\end{figure}
\end{example}

\begin{proof}
The values of the constant term functions of~$\U$ are $u$ and~$v$.
To simplify the proof, we will add $u$ and $v$ to the signature of~$\U$ as nullary operations;
this has no effect on $\Var(\U)$, up to term equivalence.

We will use a construction of Hedrl\'{\i}n and Pultr~\cite{HP}.
Given a directed graph $\G = \langle G; r \rangle$,
define the algebra $\G^* = \langle G \cup r \cup \{u,v\} ; f_0, f_1, u, v \rangle$, where
\begin{alignat*}{4}
f_0(g) &= u, &\qquad f_0((g_0,g_1)) &= g_0, &\qquad f_0(u) &= v, &\qquad f_0(v) &= v,\\
f_1(g) &= v, &\qquad f_1((g_0,g_1)) &= g_1, &\qquad f_1(u) &= u, &\qquad f_1(v) &= u,
\end{alignat*}
for all $g\in G$ and $(g_0,g_1) \in r$. (We assume that $u,v \notin G \cup r$.)
See Figure~\ref{fig: u} for an example of the algebra $\G^*$ constructed from
a directed graph~$\G$.

Each one-generated subalgebra of $\G^*$ is a homomorphic image of a subalgebra of~$\U$,
and therefore belongs to $\Var(\U)$. Thus $\G^*$ satisfies all one-variable equations
that are true in~$\U$. Since each constant term function of~$\U$ has value $u$ or~$v$,
the two-variable equations $t_1(x)\approx t_2(y)$ that are true in~$\U$ follow from
one-variable equations of the form $t(x)\approx u$ and $t(x)\approx v$.
Hence $\G^* \in \Var(\U)$.

Hedrl\'{\i}n and Pultr~\cite{HP} showed that there is a bijection between
$\hom(\G,\Hb)$ and $\hom(\G^*,\Hb^*)$. It follows immediately that $\Pb_{\cat G}$
order-embeds into~$\Lb_\U$.
\end{proof}

To contrast with the previous example, we finish this section by describing
the lattice~$\Lb_\A$, for each finite monounary algebra $\A = \langle A; f\rangle$.
We say that a non-empty subset $\{a_0,a_1,\dots,a_{n-1}\}$ of~$A$ is a \emph{cycle} if
$f(a_i)=a_{i+1\pmod n}$. For each $k\in \mathbb N$, we use $\mathbf k$
to denote the $k$-element chain.

\begin{example}\label{ex:monounary}
Let $\A = \langle A; f\rangle$ be a finite monounary algebra and
let $n$ be the least common multiple of the sizes of the cycles of\/~$\A$.
If\/ $n = 1$, then\/ $\Lb_\A \cong \1$. Otherwise, let $p_1^{k_1}\dotsm p_\ell^{k_\ell}$
be the prime decomposition of~$n$. Then
\[
 \Lb_\A \cong (\mathbf{k}_1 \sqcup \dots \sqcup \mathbf{k}_\ell) \oplus \1,
\]
where the coproduct is taken in the variety $\cat D$ of distributive lattices.
\end{example}

\begin{proof}
If $n = 1$, then every algebra in $\Var(\A)\fin$ has a trivial subalgebra,
and so $|L_\A| = 1$. Now assume that $n \ge 2$.
Let $D_n$ be the set of positive divisors of~$n$, and let
$\D_n = \langle D_n; \mathrm{lcm}, \mathrm{gcd} \rangle$ be the divisor lattice of~$n$,
where $a \le b$ in $\D_n$ if and only if $a$ divides~$b$. Then we have
$\D_n \cong (\mathbf{k}_1\oplus \1) \times \dots \times (\mathbf{k}_\ell\oplus \1)$.

Let $\Up^+(\D_n)$ denote the lattice of all non-empty up-sets of~$\D_n$, ordered by inclusion.
We start by showing that $\Lb_\A$ is isomorphic to $\Up^+(\D_n)$.

For each finite monounary algebra $\B$, define
\[
 \Cyc(\B) := \{\, |C| : \text{$C$ is a cycle of $\B$} \,\}.
\]
For all $\B \in \Var(\A)\fin$, we have
$\Cyc(\B) \subseteq D_n$, as there exists $m\in \mathbb{N}$ such that
$\A$ satisfies the equation $f^m(x) \approx f^{m+n}(x)$. For all $\B_1, \B_2 \in \Var(\A)\fin$,
we have $\B_1 \to \B_2$ if and only if $\Cyc(\B_1) \subseteq \uarrow \Cyc(\B_2)$ in~$\D_n$.
It follows that we can define an order-embedding $\psi \colon \Lb_\A \to \Up^+(\D_n)$ by
\[
 \psi(\B/{\equiv}) := \uarrow \Cyc(\B)
\]
To see that the map $\psi$ is surjective, let $\upset \in \Up^+(\D_n)$.
Choose a finite monounary algebra $\B$ that is a disjoint union of cycles and
satisfies $\Cyc(\B) = \upset$. Since each equation true in~$\A$ is of the form
$f^j(x) \approx f^{j+kn}(x)$, for some $j,k\in \mathbb{N} \cup \{0\}$,
it follows that $\B \in \Var(\A)\fin$. Clearly, we have $\psi(\B/{\equiv}) = \upset$.
Hence $\Lb_\A$ is isomorphic to the lattice $\Up^+(\D_n)$.

The lattice $\D_n$ is self-dual, and therefore $\Lb_\A$ is also isomorphic to the lattice
$\properdownsets{\D_n}$ of all non-empty down-sets of~$\D_n$, ordered by inclusion.
Using Priestley duality for the variety $\cat D_{01}$ of bounded distributive lattices
(see~\cite{DP:ilo}), we first describe the lattice $\downsets {\D_n}$ of all down-sets of~$\D_n$:
\begin{align*}
\downsets{\D_n}
 &\cong \downsets{(\mathbf{k}_1\oplus \1) \times \dots \times (\mathbf{k}_\ell\oplus \1)} \\
 &\cong \downsets{\mathbf{k}_1\oplus \1} \sqcup_{01} \dots \sqcup_{01} \downsets{\mathbf{k}_\ell\oplus \1} \\
 &\cong (\1\oplus\mathbf{k}_1\oplus \1) \sqcup_{01} \dots \sqcup_{01} (\1\oplus\mathbf{k}_\ell\oplus \1) \\
 &\cong \1\oplus(\mathbf{k}_1 \sqcup \dots \sqcup \mathbf{k}_\ell) \oplus \1,
\end{align*}
where $\sqcup_{01}$ denotes coproduct in~$\cat D_{01}$ and $\sqcup$~denotes coproduct in~$\cat D$. Hence
\[
\Lb_\A \cong \properdownsets{\D_n} \cong (\mathbf{k}_1 \sqcup \dots \sqcup \mathbf{k}_\ell) \oplus \1,
\]
as claimed.
\end{proof}

Note that, if a finite unary algebra $\A$ has no constant term functions, then coproduct in $\Var(\A)$
is disjoint union; so product distributes over coproduct, and therefore $\Lb_\A$ is distributive.
(In fact, such a variety has a natural exponentiation, and therefore
$\Lb_\A$~is relatively pseudocomplemented; see~\cite{NPT}.)
However, the homomorphism lattice induced by a finite unary algebra is not necessarily distributive.
Using Corollary~\ref{cor: retract} and Example~\ref{ex: retracts}(i), each finite partition lattice
arises as~$\Lb_\A$, for some finite unary algebra~$\A$.

\begin{remark}
While we will not be making use of cores in this paper, they serve as natural
representatives for the elements of the homomorphism lattice~$\Lb_\A$,
for a finite algebra~$\A$. So the lattice $\Lb_\A$ is finite if and only if
there is a finite bound on the sizes of the cores in $\Var(\A)\fin$.

A finite algebra $\C$ is a \emph{core} if every endomorphism of~$\C$ is an automorphism.
For each finite algebra~$\B$, there is a retraction $\phi \colon \B \to \C$ such that $\C$ is a core
(unique up to isomorphism). If we let $\cat C$ consist of one copy (up to isomorphism)
of each core in $\Var(\A)\fin$, then $\cat C$ is a transversal of the equivalence classes of~$\Lb_\A$.
Moreover, the cores are precisely the algebras that are minimal-sized within their equivalence classes.
\end{remark}

\section{The homomorphism lattice induced by a quasi-primal algebra}\label{sec: Q to L}

This section focuses on the homomorphism lattice $\Lb_\Q$ in the case that $\Q$ is a
quasi-primal algebra. We give a simple description of the lattice~$\Lb_\Q$ that could be converted
into an algorithm for computing this lattice. The description will play a pivotal role in
Section~\ref{sec: L to Q}, where we show that each finite distributive lattice can
be obtained as~$\Lb_\Q$, for some quasi-primal algebra~$\Q$.

A finite algebra $\Q$ is \emph{quasi-primal} if the
\emph{ternary discriminator} operation $\tau$ is a term function, where
\[
\tau(x,y,z) :=
 \begin{cases}
 x &\text{if $x \neq y$,}\\
 z &\text{if $x = y$.}
 \end{cases}
\]
(This implies that every finitary operation on~$Q$ that
preserves the partial automorphisms of~$\Q$ is a term function;
see Pixley~\cite{P71} and Werner~\cite{W78}).
We will use the following two general results about quasi-primal algebras.

\begin{theorem}[Pixley~{\cite[Theorem~5.1]{P70}}]\label{th: simple}
A finite algebra $\Q$ is quasi-primal if and only if
every non-trivial subalgebra of\/~$\Q$ is simple and the variety\/
$\Var(\Q)$ is both congruence permutable and congruence distributive.
\end{theorem}

\begin{theorem}[Pixley~{\cite[Theorem~4.1]{P70}}]\label{th: prod of sub}
Let\/ $\Q$ be a quasi-primal algebra. Then every finite algebra in
$\Var(\Q)$ is isomorphic to a product of subalgebras of\/~$\Q$.
\end{theorem}

We can now give our description of the lattice~$\Lb_\Q$. For an ordered set~$\Pb$,
we again use $\downsets \Pb$ to denote the lattice of all down-sets of~$\Pb$, ordered by inclusion.
We use $\Sub(\A)$ to denote the set of all subalgebras of an algebra~$\A$.

\begin{theorem}\label{th: iso}
Let $\Q$ be a quasi-primal algebra. Define the ordered set
\[
\Pb := \langle \Sub(\Q)/{\equiv}; \to \rangle
\]
and let\/ $\overline \Pb$ denote $\Pb$ without its top.
\begin{enumerate}[label={\upshape(\roman*)}]
\item If\/ $\Q$ has no trivial subalgebras, then $\Lb_\Q \cong \downsets \Pb$.
\item If\/ $\Q$ has a trivial subalgebra, then $\Lb_\Q \cong \downsets {\overline{\Pb}}$.
\end{enumerate}
\end{theorem}

\begin{proof}
To simplify the notation, let $\mathscr P$ be a transversal of $\Sub(\Q)/{\equiv}$
and define the ordered set $\Pb := \langle \mathscr P; \to \rangle$.
Without loss of generality, we can assume that $\Q \in \mathscr P$, whence
$\Q$ is the top element of~$\Pb$.

We prove the theorem via a sequence of four claims.

\begin{claim}\label{clm: up-set prod}
For each non-trivial finite algebra $\A \in \Var(\Q)$, there exists a non-empty
up-set\/~$\upset$ of\/~$\Pb$ such that $\A \equiv \prod \upset$.
\end{claim}

\begin{proofofclaim}{\ref{clm: up-set prod}}
Let $\A$ be a non-trivial finite algebra in $\Var(\Q)$.
By Theorem~\ref{th: prod of sub}, we know that $\A$ is isomorphic to a product
$\prod_{i \in I} \B_i$ of subalgebras of~$\Q$, for some non-empty set~$I$.
For each $i \in I$, there exists $\Q_i \in \mathscr P$ such that $\Q_i \equiv \B_i$.
Now define the up-set $\upset$ of~$\Pb$ by
\[
\upset := \uarrow \{\, \Q_i \mid i \in I \,\}.
\]
For each $i\in I$, we have $\Q_i \in \upset$ with $\Q_i \to \B_i$.
Thus $\prod \upset \to \prod_{i \in I} \B_i$. For each $\C \in \upset$,
there exists $i \in I$ such that $\Q_i \to \C$ and so $\B_i \to \C$.
Thus $\prod_{i \in I}\B_i \to \prod \upset$. We have shown that
$\A \cong \prod_{i \in I} \B_i \equiv \prod \upset$, as required.
\end{proofofclaim}

\begin{claim}\label{clm: upP into LsubQ}
Let $\Up^+(\Pb)$ denote the set of all non-empty up-sets of\/~$\Pb$.
Then there is an order-embedding $\sigma \colon \langle \Up^+(\Pb) ; \supseteq \rangle \to \Lb_\Q$
given by $\sigma(\upset) = \bigl(\prod \upset\bigr)/{\equiv}$.
\end{claim}

\begin{proofofclaim}{\ref{clm: upP into LsubQ}}
Note that $\sigma$ is well defined, as each up-set $\upset$ of~$\Pb$ is a finite set
of subalgebras of~$\Q$. To see that $\sigma$ is order-preserving, let
$\upset_1, \upset_2 \in \Up^+(\Pb)$ with $\upset_1 \supseteq \upset_2$.
Then $\prod \upset_1 \to \prod \upset_2$, via a projection, which gives
$\sigma(\upset_1) \to \sigma(\upset_2)$.

To prove that $\sigma$ is an order-embedding, let
$\upset_1, \upset_2 \in \Up^+(\Pb)$ with $\sigma(\upset_1) \to \sigma(\upset_2)$.
Then there is a homomorphism $\alpha \colon \prod \upset_1 \to \prod \upset_2$.
We want to show that $\upset_1 \supseteq \upset_2$, so let $\A \in \upset_2$.
Then there is a projection $\pi \colon \prod \upset_2 \to \A$.
Thus, we obtain a homomorphism $\eta \colon \prod \upset_1 \to \A$, where
$\eta = \pi \circ \alpha$. Define $\B := \Image(\eta) \le \A$.
To prove that $\A \in \upset_1$, we consider two cases.

\Case{Case \textup{(}a\textup{):} $\B$ is trivial.}
Then $\A$ has a trivial subalgebra, so $\Q \to \A$, via a constant map.
Since $\A \in \mathscr P$ and $\Q$ is the top element of~$\Pb$, we have
$\A = \Q \in \upset_1$, as required.

\Case{Case \textup{(}b\textup{):} $\B$ is non-trivial.}
Since the subalgebra $\B$ of~$\Q$ is simple and $\Var(\Q)$ is congruence distributive
(see Theorem~\ref{th: simple}), we can apply the finite-product version of
J\'{o}nsson's Lemma to the surjection $\eta \colon \prod \upset_1 \to \B$.
For some $\C \in \upset_1$, there is a projection $\rho \colon \prod \upset_1 \to \C$
and a homomorphism $\eta^\flat \colon \C \to \B$ such that the following diagram commutes.

\begin{center}
\begin{tikzpicture}[node distance=1.5cm]
\node (a) {$\prod \upset_1$};
\node (b) [right of = a] {$\B$};
\node (c) [below of = a] {$\C$};
\node (d) [right of = b] {$\A$};
\draw[->>] (a) -- node[above]{$\eta$} (b);
\draw[->] (a) -- node[left]{$\rho$} (c);
\draw[->] (c) -- node[below right]{$\eta^\flat$} (b);
\draw[>->] (b) -- (d);
\end{tikzpicture}
\end{center}

We have found $\C \in \upset_1$ such that $\C \to \A$. Since $\A \in \mathscr P$
and $\upset_1$~is an up-set of~$\Pb$, we have shown that $\A \in \upset_1$,
as required.
\end{proofofclaim}

\begin{claim}\label{clm: parti}
If\/ $\Q$ has no trivial subalgebras, then $\Lb_\Q \cong \downsets \Pb$.
\end{claim}

\begin{proofofclaim}{\ref{clm: parti}}
Assume that $\Q$ has no trivial subalgebras.
Let $\Up(\Pb)$ denote the set of all up-sets of~$\Pb$. We want to extend the
order-embedding $\sigma$ from Claim~\ref{clm: upP into LsubQ} to an order-isomorphism
$\sigma^\sharp \colon \langle \Up(\Pb) ; \supseteq \rangle \to \Lb_\Q$
given by $\sigma^\sharp(\upset) = \bigl(\prod \upset\bigr)/{\equiv}$.

For all $\upset \in \Up(\Pb)$, we have $\prod \upset \to \prod \emptyset$, via the constant
map. Thus it follows that $\sigma^\sharp$ is order-preserving. To see that $\sigma^\sharp$ is an
order-embedding, it remains to show that $\sigma^\sharp(\emptyset) \not\to \sigma^\sharp(\upset)$,
for all $\upset \in \Up^+(\Pb)$.

Let $\upset \in \Up^+(\Pb)$ and suppose that $\sigma^\sharp(\emptyset) \to \sigma^\sharp(\upset)$.
Then $\prod \emptyset \to \prod \upset$. Since the up-set $\upset$ is
non-empty, it contains the top element $\Q$ of~$\Pb$.
So $\prod \upset \to \Q$, via a projection, and thus $\prod \emptyset \to \Q$.
But this implies that $\Q$ has a trivial subalgebra, which is a contradiction.
Hence $\sigma^\sharp$ is an order-embedding.

The map $\sigma^\sharp$ is surjective, by Claim~\ref{clm: up-set prod}, since
$\sigma^\sharp(\emptyset) = \1/{\equiv}$, for any trivial algebra $\1 \in \Var(\Q)$.
Hence $\Lb_\Q \cong \langle \Up(\Pb) ; \supseteq \rangle \cong \downsets \Pb$.
\end{proofofclaim}

\begin{claim}\label{clm: partii}
If\/ $\Q$ has a trivial subalgebra, then $\Lb_\Q \cong \downsets {\overline{\Pb}}$,
where $\overline \Pb$ denotes the ordered set\/ $\Pb$ without its top.
\end{claim}

\begin{proofofclaim}{\ref{clm: partii}}
Let $\1$ be a trivial subalgebra of~$\Q$. Note that $\1 \equiv \Q$.
By Claim~\ref{clm: upP into LsubQ}, there is an order-embedding
$\sigma \colon \langle \Up^+(\Pb) ; \supseteq \rangle \to \Lb_\Q$
given by $\sigma(\upset) = \bigl(\prod \upset\bigr)/{\equiv}$.
The map $\sigma$ is surjective, by Claim~\ref{clm: up-set prod},
since $\{\Q\}$ is an up-set of $\Pb$ with $\prod \{\Q\} \cong \Q \equiv \1$.
Hence we have $\Lb_\Q \cong \langle \Up^+(\Pb) ; \supseteq \rangle \cong \downsets {\overline\Pb}$.
\end{proofofclaim}

Claims~\ref{clm: parti} and~\ref{clm: partii} complete the proof of the theorem,
as $\Pb = \langle \mathscr P; \to \rangle$ is isomorphic to
$\langle \Sub(\Q)/{\equiv}; \to \rangle$.
\end{proof}

It follows immediately from Theorem~\ref{th: iso} that $\Lb_\Q$ is a finite distributive
lattice, for each quasi-primal algebra~$\Q$. Moreover, since the ordered set $\Pb$ always has a top,
it also follows from the theorem that, if $\Q$ has no trivial subalgebras,
then the lattice $\Lb_\Q$ has a join-irreducible top.

\begin{remark}
For a quasi-primal algebra $\Q$ with no trivial subalgebras, the variety $\Var(\Q)$
satisfies a stronger property that implies the lattice $\Lb_\Q$ is distributive.
We can see this using the full duality for $\Var(\Q)$ described by Davey and
Werner~\cite[Section 2.7]{DW:first} (based on the duality given by Keimel and Werner~\cite{KW}).

If $\Q$ has no trivial subalgebras, then there is a dual equivalence between $\Var(\Q)$
and a category~$\cat X$ of topological partial unary algebras, where the product in~$\cat X$
is direct product and the pairwise coproduct in~$\cat X$ is disjoint union.
Therefore product distributes over coproduct in~$\cat X$,
and so coproduct distributes over product in $\Var(\Q)$,
which implies that $\Lb_\Q$ is distributive.

Varieties of algebras in which coproduct distributes over product are, in a sense, more common
than those in which product distributes over coproduct; see~\cite{DW:dist}.
\end{remark}

\section{The covering forest of a finite ordered set}\label{sec: tree}

In the next section, we will prove that every finite distributive lattice arises
as a homomorphism lattice $\Lb_\Q$, for some quasi-primal algebra~$\Q$. We use
a construction, introduced by Behncke and Leptin~\cite{BL}, for converting a
finite ordered set into a forest. In this section, we establish some useful
properties of this construction, and show the relationship with the universal
covering tree from graph~theory.

A \emph{forest} is an ordered set $\F$ such that, for all $a\in F$ and all $b,c\in \uarrow a$,
we have $b \le c$ or $c \le b$. A \emph{tree} is a connected forest.
Given a finite ordered set~$\Pb$, we shall construct a finite forest~$\F$
with a surjective order-preserving map $\phi \colon \F \to \Pb$.

Recall that, given an alphabet $P$, the set of all finite words in~$P$ is denoted by~$P^*$.
For elements $a$ and~$b$ of an ordered set~$\Pb$, we write $a\prec b$ to indicate that
$a$~is covered by~$b$ in~$\Pb$. We use $\Max(\Pb)$ for the set of all maximal elements of~$\Pb$
and $\Min(\Pb)$ for the set of all minimal elements of~$\Pb$.

\begin{definition}[Behncke and Leptin~\cite{BL}]\label{def: s}
Let $\Pb$ be a finite ordered set.
\begin{enumerate}[leftmargin=2em,itemsep=0.5ex]
\item
Consider the set of all covering chains in $\Pb$ that reach a maximal element,
viewed as words in the alphabet~$P$:
\[
F := \{\, a_1 a_2 \dots a_n \in P^* \mid
  n\in \mathbb{N},\ a_1 \prec a_2 \prec \cdots \prec a_n \text{ and } a_n \in \Max(\Pb) \,\}.
\]
We define the \emph{covering forest}
of~$\Pb$ to be the ordered set $\F = \langle F ; \le \rangle$, where
\[
w \le v \iff (\exists u \in P^*)\ w = uv.
\]
(That is, we have $w \le v$ if and only if $v$ is a final segment of~$w$.)
It is easy to check that the ordered set $\F$ is indeed a forest.
\item
We can define the order-preserving surjection $\phi \colon \F \to \Pb$ by
\[
\phi(a_1a_2\dots a_n) := a_1,
\]
for each $a_1a_2\dots a_n \in F$.
\item
If the ordered set $\Pb$ has a top element, then its covering forest $\F = \langle F ; \le \rangle$
is a tree, and we refer to $\F$ as the \emph{covering tree} of~$\Pb$.
\end{enumerate}
In Figure~\ref{fig: P to S}, we see an example of the covering forest~$\F$ of an
ordered set~$\Pb$ and the map $\phi \colon F \to P$.
\end{definition}

\begin{figure}[t]
\begin{tikzpicture}[node distance=1cm]
\begin{scope}[node distance=1cm]
 \node[element,label=left: {\small 5}] (5) {};
 \node[element,above right of=5,label=left: {\small 3}] (3) {};
 \node[element,below right of=3,label=right: {\small 6}] (6) {};
 \node[element,above right of=6,label=right: {\small 4}] (4) {};
 \node[element,above left of=3,label=left: {\small 1}] (1) {};
 \node[element,above left of=4,label=right: {\small 2}] (2) {};
 \draw (1) -- (3) -- (5);
 \draw (2) -- (4) -- (6);
 \draw (2) -- (3) -- (6);
 \node[below of=3,yshift=-0.1cm] (P) {$\Pb$};
\end{scope}
\begin{scope}[xshift=5.5cm,yshift=-0.4cm,node distance=1.125cm]
  \node[element,label=left: {\small 531}] (531) {};
  \node[element,right of=531,label=left: {\small 532}] (532) {};
  \node[element,right of=532,label=right: {\small 631}] (631) {};
  \node[element,right of=631,label=right: {\small 632}] (632) {};
  \node[element,right of=632,label=right: {\small 642}] (642) {};
  \node[element,above of=532,label=left: {\small 31}] (31) {};
  \node[element,above of=631,label=left: {\small 32}] (32) {};
  \node[element,above of=642,label=right: {\small 42}] (42) {};
  \node[element,above of=31,label=left: {\small 1}] (1f) {};
  \node[element,above of=42,xshift=-1.125cm,label=right: {\small 2}] (2f) {};
  \draw (1f) -- (31) -- (531);
  \draw (31) -- (631);
  \draw (2f) -- (42) -- (642);
  \draw (2f) -- (32) -- (532);
  \draw (32) -- (632);
  \node[right of=642,xshift=0.5cm] (S) {$\F$};
  \begin{pgfonlayer}{background}
    \node[potato,fit=(1f),inner xsep=14pt,inner ysep=10pt,xshift=-3pt] {};
    \node[potato,fit=(2f),inner xsep=14pt,inner ysep=10pt,xshift=3pt] {};
    \node[potato,fit=(31)(32),inner xsep=16pt,inner ysep=10pt,xshift=-5pt] {};
    \node[potato,fit=(42),inner xsep=16pt,inner ysep=10pt,xshift=5pt] {};
    \node[potato,fit=(531)(532),inner xsep=18pt,inner ysep=10pt,xshift=-8pt] {};
    \node[potato,fit=(631)(642),inner xsep=18pt,inner ysep=10pt,xshift=8pt] {};
  \end{pgfonlayer}
\end{scope}
\node[right of=2,xshift=0.5cm] (b) {};
\node[right of=b,xshift=1cm] (a) {};
\path[->] (a) edge [bend right] node[above] {$\phi$} (b);
\end{tikzpicture}
\caption{The covering forest~$\F$ of an ordered set~$\Pb$.}\label{fig: P to S}
\end{figure}
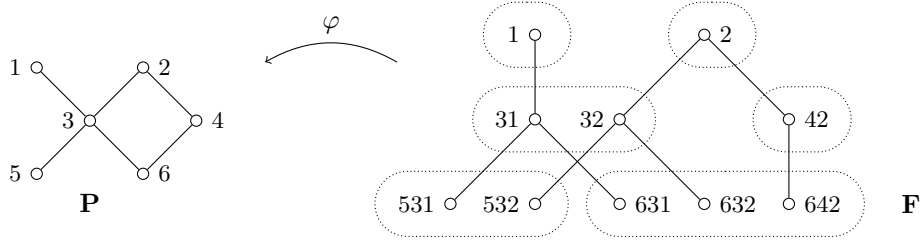

We now begin to consider the relationship between the covering forest~$\F$
of a finite ordered set~$\Pb$ and the universal covering tree from graph theory.

\begin{definition}
First define the \emph{lower-neighbourhood} of an element~$a$ of a finite ordered
set~$\A$ to be the set
\[
 N_a := \{\, c\in A \mid c \prec a \,\}.
\]
For finite ordered sets $\A$ and $\B$, we will say that $\gamma \colon A \to B$
is a \emph{covering map} if $\gamma$~restricts to a bijection
on the maximals and on lower-neighbourhoods. That~is:
\begin{enumerate}[itemsep=0.5ex]
\item $\gamma \restrictedto {\Max(\A)} \colon \Max(\A) \to \Max(\B)$ is a bijection, and
\item $\gamma \restrictedto {N_a} \colon N_a \to N_{\gamma(a)}$ is a bijection, for all $a\in A$.
\end{enumerate}
\end{definition}

\begin{remark}
This is analogous to the definitions of covering map for connected graphs and for
connected directed graphs; see \cite{JS,DHM}. Our condition~(1) replaces the condition
that $\gamma$ is surjective. In our condition~(2), we use `lower-neighbourhood' instead
of `neighbourhood' in the case of graphs, and instead of `in-neighbourhood' and
`out-neighbourhood' in the case of directed graphs.

Covering maps for ordered sets have been considered by Hoffman~\cite{MH}
in a setting that is both more special than ours (ranked ordered sets)
and more general (covers in the ordered sets are assigned positive integer weights).
\end{remark}

Before stating some useful properties of covering maps, we require a definition.

\begin{definition}
Let $\A$ and $\B$ be ordered sets. We shall say that an order-preserving map
$\alpha \colon \A \to \B$ is a \emph{quotient map} if, for all $b_1 \le b_2$ in~$\B$,
there exists $a_1\le a_2$ in~$\A$ such that $\alpha(a_1) = b_1$ and $\alpha(a_2) = b_2$.
Note that, because $\le$ is reflexive, a quotient map is necessarily surjective.
\end{definition}

\begin{lemma}\label{lem:covmaps}
Let $\gamma \colon A \to B$ be a covering map for finite ordered sets~$\A$ and\/~$\B$.
\begin{enumerate}[label={\upshape(\roman*)}]
\item The map $\gamma$ is cover-preserving \textup(and is therefore order-preserving\textup).
\item For each covering chain $b_1 \prec b_2 \prec \dots \prec b_n$ in~$\B$
  with $b_n\in \Max(\B)$, there exists a covering chain $a_1 \prec a_2 \prec \dots \prec a_n$ in~$\A$
  with $a_n\in \Max(\A)$ such that $\gamma(a_i) = b_i$, for all $i\in \{1, 2,  \dots, n\}$.
\item The map $\gamma$ is a quotient map \textup(and is therefore surjective\textup).
\end{enumerate}
\end{lemma}

\begin{proof}
Part~(i) is trivial, and part~(ii) is an easy induction: condition~(1) in the definition of
a covering map gets the induction started at $n=1$, and condition~(2) yields the inductive step.
Part~(iii) follows directly from part~(ii): let $x\le y$ in $\B$ and apply~(ii) to a covering chain
in~$\B$ that starts at~$x$, passes through~$y$ and ends at a maximal element of~$\B$.
\end{proof}

The next result follows almost immediately from the construction of the covering forest,
together with part~(iii) of the previous lemma.

\begin{lemma}\label{lem: covering}
Let\/ $\F$ be the covering forest of a finite ordered set\/~$\Pb$.
Then\/ $\phi \colon \F \to \Pb$ is a covering map \textup(and is therefore a quotient map\textup).
\end{lemma}

\begin{remark}\label{rem: cover}
It can be shown that the covering forest~$\F$ from Definition~\ref{def: s}
is the unique forest (up to isomorphism) with a covering map to~$\Pb$.
Moreover, the covering map $\phi \colon \F \to \Pb$ is the
\emph{universal cover} of~$\Pb$: for every covering map $\gamma \colon \A \to \Pb$,
there is a (necessarily unique) order-preserving map $\alpha \colon \F \to \A$
with $\phi = \gamma \circ \alpha$.
\end{remark}

We now establish some properties of the covering forest $\F$ that will be used in
the next section. In particular, we will be using the fact that, if two elements
of~$F$ are identified by the covering map~$\phi$, then the corresponding
principal down-sets of~$\F$ are order-isomorphic.

\begin{definition}\label{def: psi}
Let $\F$ be the covering forest of a finite ordered set~$\Pb$.
Then, since $\phi \colon \F \to \Pb$ is a covering map, for each $x \in F$,
we can define the bijection $\psi_x  \colon N_{\phi(x)} \to N_x$
to be the inverse of the bijection $\phi \restrictedto {N_x} \colon N_x \to N_{\phi(x)}$.
\end{definition}

\begin{lemma}\label{lem: mu}
Let\/ $\F$ be the covering forest of a finite ordered set\/~$\Pb$,
and let $u, v \in F$ with $\phi(u) = \phi(v)$. Then there is an order-isomorphism
$\mu \colon \dnarrow u \to \dnarrow v$ such that
\begin{enumerate}[label={\upshape(\roman*)}]
\item $\phi\circ \mu = \phi \restrictedto {\dnarrow u}$, and
\item $\mu \circ \psi_x = \psi_{\mu(x)}$, for all $x \in \dnarrow u$.
\end{enumerate}
\end{lemma}

\begin{proof}
Recall that the elements of~$F$ are covering chains in $\Pb$ that reach a maximal,
and that each element of~$\dnarrow u$ is of the form~$su$, for some $s \in P^*$.
We want to define the map $\mu \colon \dnarrow u \to \dnarrow v$ by
\[
\mu(su) := sv
\]
for all $su \in \dnarrow u$. Since $\phi(u) = \phi(v)$, the covering chains $u$ and $v$ start
at the same element of~$P$, so it follows that $sv \in F$, as required.

To see that $\mu$ is order-preserving, let $y \le z$ in $\dnarrow u$. Then $z = su$,
for some $s \in P^*$, and $y = tsu$, for some $t \in P^*$. Thus $\mu(y) = tsv \le sv = \mu(z)$,
whence $\mu$ is order-preserving. By symmetry, there is an order-preserving map
$\nu \colon \dnarrow v \to \dnarrow u$ such that $\nu \circ \mu = \id_{\dnarrow u}$
and $\mu \circ \nu = \id_{\dnarrow v}$.
Therefore $\mu$ and $\nu$ are mutually inverse order-isomorphisms.

For~(i), let $x \in \dnarrow u$. Then $x = su$ and $\mu(x) = sv$, for some $s \in P^*$.
If $s$ is not the empty word, then clearly $\phi(\mu(x)) = \phi(x)$. If $s$ is the empty word,
then we still have $\phi(\mu(x)) = \phi(x)$, since $\phi(u) = \phi(v)$.

For~(ii), let $x \in \dnarrow u$ and define $y := \mu(x)$. Since $\mu$
is an order-isomorphism, we have $\mu(N_x) = N_y$. Using~(i), we obtain
$\phi\restrictedto {N_y}\circ \mu \restrictedto {N_x} = \phi \restrictedto {N_x}$.
Therefore
\[
 \psi_x = (\phi \restrictedto {N_x})^{-1}
   = (\phi\restrictedto {N_y}\circ \mu \restrictedto {N_x})^{-1}
   = (\mu \restrictedto {N_x})^{-1} \circ (\phi\restrictedto {N_y})^{-1}
   = (\mu \restrictedto {N_x})^{-1} \circ \psi_y,
\]
from which (ii) follows easily.
\end{proof}

\section{Obtaining each finite distributive lattice}\label{sec: L to Q}

In Section~\ref{sec: Q to L} we saw that, for each quasi-primal algebra $\Q$, the homomorphism
lattice $\Lb_\Q$ is a finite distributive lattice. In this section, we prove that all
finite distributive lattices arise in this way.

\begin{theorem}\label{th: PtoQ}
For each finite distributive lattice $\Lb$, there exists a quasi-primal algebra $\Q$
such that\/ $\Lb_\Q$ is isomorphic to $\Lb$.
\end{theorem}

We start by proving a special case of a result of Birkhoff and Frink~\cite{BF}, since it
serves as motivation for the approach used in the proof of Theorem~\ref{th: PtoQ}.
We see that, given a finite join-semilattice~$\Sb$, we can construct an algebra~$\A$
on the same universe as~$\Sb$ such that the subalgebras of~$\A$ correspond to the principal
ideals of~$\Sb$. Recall that we use $\Sub(\A)$ to denote the set of all subalgebras of~$\A$.

\begin{lemma}[Birkhoff and Frink~\cite{BF}]\label{lem: frink}
For each finite join-semilattice $\Sb$, there exists a finite algebra $\A$ such that
$\Sub(\A)$ is order-isomorphic to~$\Sb$, where $\Sub(\A)$ is ordered by inclusion.
\end{lemma}

\begin{proof}
Let $\Sb = \langle S; \vee \rangle$ be a finite semilattice.
For each covering pair $s \prec t$ in $\Sb$, define the unary operation
$f_{ts} \colon S \to S$ by
\[
f_{ts}(x) :=
 \begin{cases}
 s &\text{if $x = t$,}\\
 x &\text{otherwise.}
 \end{cases}
\]
Now define $F := \{\, f_{ts} \mid s \prec t \text{ in } \Sb \,\}$ and $\A := \langle S; \vee, F \rangle$.

Since $\Sb$ is finite, a subset $B$ of $A$ is closed under all $f_{ts} \in F$ if and only if
$B$ is a down-set. It follows that the subalgebras of~$\A$ correspond precisely
to the ideals of~$\Sb$. For each $x \in S$, let $\A_x$ denote the subalgebra of~$\A$
with universe~$\dnarrow x$. Then, again since $\Sb$ is finite, we have
$\Sub(\A) = \{\, \A_x \mid x \in S \,\}$. We can now define the order-isomorphism
$\alpha \colon \Sb \to \Sub(\A)$ by $\alpha(x):= \A_x$.
\end{proof}

For our result, we are starting with a finite distributive lattice $\Lb$. By Birkhoff's
Representation Theorem (see~\cite{DP:ilo}), there is a finite ordered set $\Pb$ such that
$\Lb$ is isomorphic to the lattice $\downsets \Pb$ of all down-sets of~$\Pb$.
Let $\Pb^\top$ denote the ordered set obtained from~$\Pb$ by adding a new top element~$\top$.

We want to construct a quasi-primal algebra $\Q$, with a trivial subalgebra, such that
$\langle \Sub(\Q)/{\equiv}; \to \rangle$ is isomorphic to~$\Pb^\top$.
By Theorem~\ref{th: iso}(ii), we will then have
\[
\Lb_\Q \cong \downsets {\overline {\Pb^\top}} = \downsets \Pb \cong \Lb,
\]
as required.

To mimic the construction from the proof of the previous lemma, the general idea
is to build the algebra $\Q$ on a semilattice $\Sb$ that will
correspond to $\langle \Sub(\Q) ; \subseteq \rangle$. There is an order-preserving map from
$\langle \Sub(\Q) ; \subseteq \rangle$ onto $\langle \Sub(\Q)/{\equiv} ; \to \rangle$.
Therefore, we need an order-preserving map from the semilattice $\Sb$ onto~$\Pb^\top$.
We will take $\Sb$ to be the covering tree of~$\Pb^\top$, as given in Definition~\ref{def: s}.

Note that this general idea will need tweaking to avoid problems with trivial subalgebras.

\begin{assumptions}\label{setup}
Throughout the rest of this section, we fix a non-trivial finite distributive lattice~$\Lb$,
with $\Lb \cong \downsets \Pb$. Let~$\Pb^\top$ be the ordered set obtained from~$\Pb$
by adding a new top element~$\top$, and let $\Sb$ be the covering tree of~$\Pb^\top$.
By Lemma~\ref{lem: covering}, we have a quotient map $\phi \colon \Sb \to \Pb^\top$.
Note that $\top$ is also the top element of~$\Sb$.
\end{assumptions}

We will construct a quasi-primal algebra $\Q$ such that there is a quotient map
$\eta \colon \Sb \to \langle \Sub(\Q)/{\equiv}; \to \rangle$ with $\ker(\eta) = \ker(\phi)$.
It will then follow that the two ordered sets $\Pb^\top$ and $\langle \Sub(\Q)/{\equiv}; \to \rangle$
are isomorphic, as required.

To avoid problems with trivial subalgebras, we will not use the universe of~$\Sb$
as the universe of~$\Q$, but will make a slight modification.

\begin{definition}
Let $\Pb^\sharp$ be the ordered set obtained from~$\Pb^\top$ by adding a
new minimal element~$m'$ below each minimal element~$m$ of~$\Pb^\top$,
so that $m$ is the unique upper cover of~$m'$.

Now let $\Sb^\sharp$ be the covering tree of~$\Pb^\sharp$.
By Lemma~\ref{lem: covering}, we have a covering map
$\phi^\sharp \colon \Sb^\sharp \to \Pb^\sharp$.
See Figure~\ref{fig: aug tree} for an example of this construction,
where $\Pb$ is the six-element ordered set from Figure~\ref{fig: P to S}.

We can easily see that $\Sb$ is a subordered set of~$\Sb^\sharp$, with
\[
 S = S^\sharp \setminus \Min(\Sb^\sharp)
 \quad\text{and}\quad
 \phi = \phi^\sharp \restrictedto S.
\]
Since $\Sb^\sharp$ is a finite tree, each minimal element~$x$ of~$\Sb^\sharp$
has a unique upper cover, which we denote by $x^{\uparrow}$.
\end{definition}

\begin{figure}[t]
\begin{tikzpicture}[node distance=1cm]
\begin{scope}[node distance=1cm]
 \node[element,label=left: {\small 5}] (5) {};
 \node[element,above right of=5,label=left: {\small 3}] (3) {};
 \node[element,below right of=3,label=right: {\small 6}] (6) {};
 \node[element,above right of=6,label=right: {\small 4}] (4) {};
 \node[element,above left of=3,label=left: {\small 1}] (1) {};
 \node[element,above left of=4,label=right: {\small 2}] (2) {};
 \node[shaded element,above right of=1,label=right: {\small $0 = \top$}] (0) {};
 \node[shaded element,below of=5,label=left: {\small 5$'$}] (5p) {};
 \node[shaded element,below of=6,label=right: {\small 6$'$}] (6p) {};
 \draw (0) -- (1) -- (3) -- (5) -- (5p);
 \draw (0) -- (2) -- (4) -- (6) -- (6p);
 \draw (2) -- (3) -- (6);
 \node[right of=6p,xshift=0.125cm] (P) {$\Pb^\sharp$};
\end{scope}
\begin{scope}[xshift=4.75cm,yshift=-0.55cm,node distance=1.125cm,label distance=-1pt]
  \node[element,label=left: {\small 5310}] (531) {};
  \node[element,right of=531,xshift=0.25cm,label=left: {\small 5320}] (532) {};
  \node[element,right of=532,xshift=0.25cm,label=right: {\small 6310}] (631) {};
  \node[element,right of=631,xshift=0.25cm,label=right: {\small 6320}] (632) {};
  \node[element,right of=632,xshift=0.25cm,label=right: {\small 6420}] (642) {};
  \node[element,above of=532,label=left: {\small 310}] (31) {};
  \node[element,above of=631,label=left: {\small 320}] (32) {};
  \node[element,above of=642,label=right: {\small 420}] (42) {};
  \node[element,above of=31,label=left: {\small 10}] (1f) {};
  \node[element,above of=42,xshift=-1.25cm,label=right: {\small 20}] (2f) {};
  \node[element,below of=531,label=left: {\small 5$'$5310}] (531p) {};
  \node[element,below of=532,label=left: {\small 5$'$5320}] (532p) {};
  \node[element,below of=631,label=right: {\small 6$'$6310}] (631p) {};
  \node[element,below of=632,label=right: {\small 6$'$6320}] (632p) {};
  \node[element,below of=642,label=right: {\small 6$'$6420}] (642p) {};
  \node[element,above of=32,yshift=0.75cm,label=above: {\small $0 = \top$}] (0f) {};
  \draw (0f) -- (1f) -- (31) -- (531) -- (531p);
  \draw (31) -- (631) -- (631p);
  \draw (0f) -- (2f) -- (42) -- (642) -- (642p);
  \draw (2f) -- (32) -- (532) -- (532p);
  \draw (32) -- (632) -- (632p);
  \node[right of=642,xshift=0.325cm,yshift=-0.45cm] (S) {$\Sb^\sharp$};
  \begin{pgfonlayer}{background}
    \node[potato,fit=(0f),inner xsep=19pt,inner ysep=10pt,yshift=3pt] {};
    \node[potato,fit=(1f),inner xsep=14pt,inner ysep=10pt,xshift=-3pt] {};
    \node[potato,fit=(2f),inner xsep=14pt,inner ysep=10pt,xshift=3pt] {};
    \node[potato,fit=(31)(32),inner xsep=16pt,inner ysep=10pt,xshift=-5pt] {};
    \node[potato,fit=(42),inner xsep=16pt,inner ysep=10pt,xshift=5pt] {};
    \node[potato,fit=(531)(532),inner xsep=18pt,inner ysep=10pt,xshift=-8pt] {};
    \node[potato,fit=(631)(642),inner xsep=18pt,inner ysep=10pt,xshift=8pt] {};
    \node[potato,fit=(531p)(532p),inner xsep=22pt,inner ysep=10pt,xshift=-12pt] {};
    \node[potato,fit=(631p)(642p),inner xsep=22pt,inner ysep=10pt,xshift=12pt] {};
  \end{pgfonlayer}
\end{scope}
\node[right of=2,xshift=0.5cm] (b) {};
\node[right of=b,xshift=1cm] (a) {};
\path[->] (a) edge [bend right] node[above] {$\phi^\sharp$} (b);
\end{tikzpicture}
\caption{The covering tree~$\Sb^\sharp$ of the ordered set~$\Pb^\sharp$.}\label{fig: aug tree}
\end{figure}

We will base our quasi-primal algebra $\Q$ on the semilattice $\Sb^\sharp$ and adapt
the approach used in the proof of Lemma~\ref{lem: frink}. We want to ensure that the
non-trivial subalgebras of~$\Q$ correspond to the non-trivial principal ideals
of~$\Sb^\sharp$. However, we also need to ensure that two such subalgebras are
homomorphically equivalent if their maximum elements are identified by~$\phi$.

\begin{definition}
To define the quasi-primal algebra $\Q$ on the universe~$S^\sharp$, we first define
some families of operations on $S^\sharp$. Recall from Definition~\ref{def: psi} that,
since $\phi^\sharp \colon \Sb^\sharp \to \Pb^\sharp$ is a covering map, for each $x \in S^\sharp$,
we let $\psi_x  \colon N_{\phi^\sharp(x)} \to N_x$ denote the inverse of the
bijection $\phi^\sharp \restrictedto {N_x} \colon N_x \to N_{\phi^\sharp(x)}$.

\begin{itemize}[leftmargin=2em,itemsep=1ex]
\item[$F$:]
These operations will ensure that the subuniverses of~$\Q$ not containing~$\top$
are down-sets. For all $a \prec b$ in~$\Pb^\sharp$ with $b \ne \top$, define
$f_{ba} \colon S^\sharp \to S^\sharp$ by
\[
f_{ba}(x) :=
 \begin{cases}
 \psi_x(a) &\text{if $\phi^\sharp(x) = b$,}\\
 x &\text{otherwise.}
 \end{cases}
\]
Let $F := \bigl\{\, f_{ba} \mid a \prec b \text{ in } P^\sharp\notop \,\bigr\}$.

\item[$G$:]
These next operations will ensure that the minimal elements of~$\Sb^\sharp$
do not form trivial subalgebras of~$\Q$. For $m \in \Min(\Pb)$, define
$g_m \colon S^\sharp \to S^\sharp$ by
\[
g_m(x) :=
 \begin{cases}
 x^{\uparrow} &\text{if $\phi^\sharp(x) = m'$,}\\
 x &\text{otherwise.}
\end{cases}
\]
Let $G := \bigl\{\, g_m \mid m \in \Min(\Pb) \,\bigr\}$.

\item[$h$:]
Finally, this operation will ensure that the only non-trivial subuniverse of~$\Q$
containing $\top$ is $Q$ itself. Fix a cyclic permutation
$\lambda \colon S^\sharp \notop \to S^\sharp \notop$. Define the
binary operation $h \colon S^\sharp \times S^\sharp \to S^\sharp$ by
\[
h(x,y):=
 \begin{cases}
 \lambda(x) &\text{if $x \neq \top$ and $y = \top$,}\\
 x &\text{otherwise.}
 \end{cases}
\]
\end{itemize}
Let $\vee$ be the join-semilattice operation on the tree~$\Sb^\sharp$, and let $\tau$ denote
the ternary discriminator operation on~$S^\sharp$. We can now define the quasi-primal algebra
\[
\Q := \langle S^\sharp ; \vee, F, G, h, \tau \rangle.
\]
\end{definition}

\begin{note}\label{newnote}
We will use the fact that, for each covering pair $s \prec t$ in~$\Sb^\sharp$ with $t \ne \top$,
we have $f_{\phi^\sharp(t)\phi^\sharp(s)}(t) = s$ in~$\Q$. To see this, first note that,
since $\phi^\sharp$ is a covering map, we must have $\phi^\sharp(s) \prec \phi^\sharp(t)$
in~$\Pb^\sharp$ with $\phi^\sharp(t) \ne \top$. We now calculate
\[
 f_{\phi^\sharp(t)\phi^\sharp(s)}(t) = \psi_t\bigl(\phi^\sharp(s)\bigr)
 = \psi_t \circ \phi^\sharp\restrictedto {N_t}(s) = s,
\]
as required.
\end{note}

Throughout this section, we use $\dnarrow u$ to denote the down-set of $u$ in
the tree~$\Sb^\sharp$.

\begin{lemma}\label{lem: subuniverse}
The non-empty subuniverses of\/ $\Q$ are $\{\top\}$ and $\dnarrow u$, for all $u \in S$.
\end{lemma}

\begin{proof}
Let $A$ be a non-empty subuniverse of~$\Q$ with $A \ne \{\top\}$. We will show that
$A = \dnarrow u$, for some $u \in S$. Since $A$ is closed under~$\vee$,
there is a maximum element $u$ of~$A$ in~$\Sb^\sharp$. Suppose that $u \in \Min(\Sb^\sharp)$.
Then $\phi^\sharp(u) = m'$, for some $m\in \Min(\Pb)$.
This gives $u^{\uparrow} = g_m(u) \in A$, which is a contradiction.
Thus $u \in S$.

First assume that $u = \top$. Since $A$ is closed under~$h$, it follows that $A$
is closed under the cyclic permutation $\lambda$ of~$S^\sharp\notop$.
As $A\notop \ne \emptyset$, we have $A = Q = \dnarrow u$.

Now we can assume that $u \ne \top$, and so $\top \notin A$.
We want to show that $A$ is a down-set of~$\Sb^\sharp$.
Consider a covering pair $s \prec t$ in~$\Sb^\sharp$ with $t \ne \top$.
By Note~\ref{newnote}, we have $f_{\phi^\sharp(t)\phi^\sharp(s)}(t) =  s$.
As $A$ is closed under each $f_{ba}\in F$ and $\Sb^\sharp$ is finite,
it follows that $A$ is a down-set. Thus $A = \dnarrow u$.

It remains to check that the sets $\{\top\}$ and $\dnarrow u$, for $u \in S$,
are indeed subuniverses of~$\Q$. First consider $\{\top\}$.
This set is clearly closed under~$\vee$, $\tau$ and~$h$.
Since $\phi^\sharp(\top) = \top$, the set $\{\top\}$ is also closed under each $f_{ba} \in F$
and each $g_m\in G$. Thus $\{\top\}$ is a subuniverse of~$\Q$.

Now consider $\dnarrow u$, for some $u\in S\notop$. The set $\dnarrow u$ is closed
under each $f_{ba} \in F$, since $f_{ba}(x) \in N_x \cup \{x\}$, for all $x \in S^\sharp$.
For each $m\in \Min(\Pb)$, if $\phi^\sharp(x) = m'$, then $x \in \Min(\Sb^\sharp)$
and so $x^{\uparrow} \in \dnarrow u$, as $u \notin \Min(\Sb^\sharp)$.
Thus $\dnarrow u$ is closed under each $g_m \in G$.
As $\top \notin \dnarrow u$, the set $\dnarrow u$ is closed under~$h$.
Since $\dnarrow u$ is also closed under~$\vee$ and~$\tau$, it is a subuniverse of~$\Q$.
\end{proof}

\begin{definition}\label{def: qu}
Using the previous lemma, for each $u\in S$, we can define $\Q_u$
to be the subalgebra of~$\Q$ with universe $\dnarrow u$.
\end{definition}

\begin{lemma}\label{lem: equiv}
Each subalgebra of\/ $\Q$ is homomorphically equivalent to~$\Q_u$,
for some $u \in S$.
\end{lemma}

\begin{proof}
This follows immediately from Lemma~\ref{lem: subuniverse}, since the trivial subalgebra
of~$\Q$ is homomorphically equivalent to $\Q = \Q_\top$.
\end{proof}

\goodbreak

\begin{lemma}\label{lem: iff phi}
Let $u, v \in S$. Then $\phi(u) = \phi(v)$ if and only if\/ $\Q_u \cong \Q_v$.
\end{lemma}

\begin{proof}
First assume that there is an isomorphism $\alpha \colon \Q_u \to \Q_v$.
We can assume that $u \ne \top$, since otherwise $u = v = \top$.
As $\alpha$ preserves~$\vee$, we have an order-isomorphism
$\alpha \colon \dnarrow u \to \dnarrow v$ and so $\alpha(u) = v$.
Since $u\notin \Min(\Sb^\sharp)$, we can choose $s\in S^\sharp$ with $s \prec u$.
Using Note~\ref{newnote}, we have $f_{\phi^\sharp(u)\phi^\sharp(s)}(u) = s \ne u$.
Since $\alpha$ is an isomorphism with $\alpha(u) = v$,
this implies that $f_{\phi^\sharp(u)\phi^\sharp(s)}(v) \ne v$.
Hence $\phi^\sharp(v) = \phi^\sharp(u)$ and so $\phi(u) = \phi(v)$, as $u,v\in S$.

Now assume that $\phi(u) = \phi(v)$. We want to show that $\Q_u \cong \Q_v$.
Since we have $\phi^{-1}(\top) = \{\top\}$, we can assume that $u,v \ne \top$.
By applying Lemma~\ref{lem: mu} to the covering map $\phi^\sharp \colon \Sb^\sharp \to \Pb^\sharp$,
we obtain an order-isomorphism $\mu \colon \dnarrow u \to \dnarrow v$ such that
\begin{enumerate}[label={\upshape(\roman*)}]
\item $\phi^\sharp \circ \mu = \phi^\sharp \restrictedto {\dnarrow u}$, and
\item $\mu \circ \psi_x = \psi_{\mu(x)}$, for all $x \in \dnarrow u$.
\end{enumerate}
As $\mu$ is an order-isomorphism, it must preserve $\vee$ and~$\tau$.
To prove that $\mu \colon \Q_u \to \Q_v$ is an isomorphism,
it remains to check that $\mu$ preserves the operations
in $F \cup G$ and the operation~$h$.
\begin{itemize}[leftmargin=2em,itemsep=1ex]
\item[$F$:]
Let $a \prec b$ in $P^\sharp\notop$. We want to show that $\mu$ preserves~$f_{ba}$.
Let $x \in \dnarrow u$. By~(i), we have $\phi^\sharp(x) = \phi^\sharp(\mu(x))$.
Since $f_{ba}$ acts as the identity map on~$S^\sharp \setminus (\phi^\sharp)^{-1}(b)$,
we can assume that $\phi^\sharp(x) = \phi^\sharp(\mu(x)) = b$. Using~(ii), we have
\[
 \mu(f_{ba}(x)) = \mu(\psi_x(a)) = \psi_{\mu(x)}(a) = f_{ba}(\mu(x)).
\]
Thus $\mu$ preserves $f_{ba}$.
\item[$G$:]
Let $m \in \Min(\Pb)$. We want to show that $\mu$ preserves $g_m$. Let $x\in \dnarrow u$.
By~(i), we have $\phi^\sharp(x) = \phi^\sharp(\mu(x))$.
Since $g_m$ acts as the identity map on~$S^\sharp \setminus (\phi^\sharp)^{-1}(m')$,
we can assume that $\phi^\sharp(x) = \phi^\sharp(\mu(x)) = m'$.
Since $\mu \colon \dnarrow u \to \dnarrow v$ is an order-isomorphism, we obtain
\[
 \mu(g_m(x)) = \mu(x^{\uparrow}) = \mu(x)^{\uparrow} = g_m(\mu(x)).
\]
Thus $\mu$ preserves~$g_m$.
\item[$h$:]
We are assuming that $u,v\ne \top$. Thus the binary operation $h$ acts as the first projection
on both $\dnarrow u$ and $\dnarrow v$. Hence $\mu\colon \dnarrow u \to \dnarrow v$
preserves~$h$.
\end{itemize}
We have shown that $\mu \colon \Q_u \to \Q_v$ is an isomorphism.
\end{proof}

We can now construct the required quotient map from the covering tree $\Sb$ of $\Pb^\top$ to the
homomorphism order on $\Sub(\Q)$.

\begin{lemma}\label{lem: eta SOP}
There exists a quotient map $\eta \colon \Sb \to \langle \Sub(\Q)/{\equiv}; \to \rangle$
such that $\ker(\eta) = \ker(\phi)$.
\end{lemma}

\begin{proof}
We can define the map $\eta \colon S \to \Sub(\Q)/{\equiv}$ by
\[
 \eta(x) := \Q_x/{\equiv},
\]
for each $x\in S$, using Definition~\ref{def: qu}.

To show that $\ker(\phi) \subseteq \ker(\eta)$, let $u,v\in S$ with $\phi(u) = \phi(v)$.
Then $\Q_u \cong \Q_v$, by Lemma~\ref{lem: iff phi}. Thus $\Q_u \equiv \Q_v$,
giving $\eta(u) = \eta(v)$.

To show that $\ker(\eta) \subseteq \ker(\phi)$, let $u,v\in S$ with $\eta(u) = \eta(v)$.
Then $\Q_u \equiv \Q_v$, so $\Q_u \to \Q_v$ and $\Q_v \to \Q_u$. Since $\Q$ is quasi-primal,
each of these homomorphisms is either constant or an embedding (see Theorem~\ref{th: simple}).
If both are embeddings, then $\Q_u \cong \Q_v$ and therefore $\phi(u) = \phi(v)$,
by Lemma~\ref{lem: iff phi}. Without loss of generality, we can now consider the case that
there is a constant homomorphism $\Q_u \to \Q_v$, whence $\Q_v$ has a trivial subalgebra.
Since $\Q_v \to \Q_u$, it follows that $\Q_u$ also has a trivial subalgebra.
By Lemma~\ref{lem: subuniverse}, the only trivial subuniverse of~$\Q$ is $\{\top\}$.
So we can conclude that $u = v = \top$. We have shown that $\ker(\eta) = \ker(\phi)$.

To see that $\eta$ is order-preserving, let $u \le v$ in~$\Sb$. Then
$\dnarrow u \subseteq \dnarrow v$ in~$\Sb^\sharp$ and therefore $\Q_u \to \Q_v$,
via the inclusion. Hence $\eta(u) \to \eta(v)$.

To prove that $\eta$ is a quotient map, let $\A, \B \in \Sub(\Q)$ with $\A \to \B$.
We want to find $x \le y$ in $\Sb$ such that $\Q_x \equiv \A$ and $\Q_y \equiv \B$.
By Lemma~\ref{lem: equiv}, there exist $u,v\in S$ with $\Q_u \equiv \A$ and $\Q_v \equiv \B$.
Since $u \le \top$ in $\Sb$, we can assume that $v \neq \top$.
As $\A \to \B$, we have a homomorphism $\Q_u \to \Q_v$. Since $v \neq \top$,
the algebra $\Q_v$ has no trivial subalgebras, so this homomorphism is an embedding.
Therefore, using Lemma~\ref{lem: subuniverse}, we must have
\[
\A \equiv \Q_u \cong \Q_w \le \Q_v \equiv \B,
\]
for some $w \in S$ with $w \le v$, as required. Hence $\eta$ is a quotient map.
\end{proof}

We are now able to prove Theorem~\ref{th: PtoQ} by showing that $\Lb_\Q \cong \Lb$.
Recall that $\Lb \cong \downsets \Pb$, from Assumptions~\ref{setup}.
Since we have quotient maps
\[
 \phi \colon \Sb \to \Pb^\top \quad\text{and}\quad
 \eta \colon \Sb \to \langle \Sub(\Q)/{\equiv}; \to \rangle
\]
with $\ker(\phi) = \ker(\eta)$, it follows that
$\Pb^\top \cong \langle \Sub(\Q)/{\equiv}; \to \rangle$.
Since $\Q$ has a trivial subuniverse $\{\top\}$, we can use
Theorem~\ref{th: iso}(ii) to obtain
\[
\Lb_\Q \cong \downsets {\overline {\Pb^\top}} = \downsets \Pb \cong \Lb,
\]
which completes the proof of Theorem~\ref{th: PtoQ}.

\section{The class of homomorphism lattices}\label{sec: extra}

In this section, we consider the class $\Lhom$ consisting of all lattices
$\Lb$ such that $\Lb \cong \Lb_\A$, for some finite algebra~$\A$.
We pose the problem:
\begin{quote}
Does every finite lattice belong to $\Lhom$?
\end{quote}
By Theorem~\ref{th: PtoQ}, we know that every finite distributive lattice belongs to~$\Lhom$.

\begin{lemma}
The class $\Lhom$ is closed under finite products.
\end{lemma}
\begin{proof}
Consider finite algebras $\A_1 = \langle A_1; F_1\rangle$ and $\A_2 = \langle A_2; F_2\rangle$.
Up to term equivalence, we can assume that $F_1$ and $F_2$ are disjoint
and do not contain nullary operations. So we can define algebras $\B_1$ and $\B_2$
of signature $F_1 \cup F_2 \cup \{\ast\}$ such that $\B_i$ is term equivalent to~$\A_i$,
where $\ast$ is a binary operation that acts as the first projection on $\B_1$
and the second projection on~$\B_2$.

Now the algebra $\C := \B_1 \times \B_2$ is the independent product of $\B_1$ and~$\B_2$,
and it follows that $\Lb_\C \cong \Lb_{\B_1} \times \Lb_{\B_2}$; see~\cite{GLP}.
\end{proof}

It is not clear whether $\Lhom$ is closed under forming homomorphic images or
taking sublattices. Indeed, if we could show that $\Lhom$ were closed under taking sublattices,
then it would follow from Example~\ref{ex: retracts}(i) below that $\Lhom$ contained all finite lattices.

\subsection*{Congruence lattices}

One of the most famous unsolved problems in universal algebra,
the \emph{Finite Lattice Representation Problem},
asks whether every finite lattice arises as the congruence lattice of a finite algebra;
see Problem~13 of Gr\"atzer~\cite[p.~116]{GUAbook}. It is therefore natural to ask
whether the congruence lattice of each finite algebra belongs to~$\Lhom$. That is,
given a finite algebra $\A$, does there exist a finite algebra $\B$ such that $\Lb_\B \cong \Con(\A)$?

We can obtain some interesting applications by focussing on the special case where
$\Lb_\A \cong \Con(\A)$. We can assume, without affecting the lattice $\Con(\A)$,
that every element of~$\A$ is the value of a nullary term function. Then,
for all $\B\in \Var(\A)$, the values of the nullary term functions of $\B$ form a
subalgebra that we will denote by~$\B_0$.

\begin{lemma}\label{lem: new}
Let $\A$ be a finite algebra such that each element is the value of a
nullary term function. Then the following are equivalent:
\begin{enumerate}[label={\upshape(\roman*)}]
\item
$\Lb_\A \cong \Con(\A)$;
\item
for every finite algebra $\B\in \Var(\A)$, we have $\B \to \B_0$;
\item
\begin{enumerate}[label={\upshape(\alph*)}]
\item
for every finite subdirectly irreducible algebra $\B\in \Var(\A)$, we have $\B \to \B_0$, and
\item
for all\/ $\theta_1, \theta_2\in \Con(\A)$, we have
$(\A/\theta_1) \times (\A/\theta_2) \to \A/(\theta_1\cap \theta_2)$.
\end{enumerate}
\end{enumerate}
\end{lemma}

\begin{proof}
We will be using the following consequence of our assumption that each element of~$\A$ is
named via a nullary term function:
\begin{itemize}
\item[($\natural$)]
For all $\theta \in \Con(\A)$, if $\C \equiv \A/\theta$ and $\B$ embeds into $\C$, then $\B \equiv \A/\theta$.
\end{itemize}
First we define the map $\psi \colon \Con(\A) \to \Lb_\A$ by
\[
 \psi(\theta) := (\A/\theta)/{\equiv},
\]
for all $\theta\in \Con(\A)$. Note that $\psi$ is order-preserving,
and that it follows that $\psi$ is an order-embedding as every element of~$\A$
is named. Hence, since $\Con(\A)$ is finite, we have $\Lb_\A \cong \Con(\A)$
if and only if the map $\psi$ is surjective.

(i)~$\Rightarrow$~(ii): Assume that (i) holds and let $\B\in \Var(\A)_{\mathrm{fin}}$.
The map $\psi$ is surjective, so $\B \equiv \A/{\theta}$, for some $\theta \in \Con(\A)$.
Since $\B_0$ embeds into $\B$, it follows by~$(\natural$) that $\B_0 \equiv \A/{\theta}$.
Thus $\B \equiv \B_0$, whence (ii)~holds.

(ii)~$\Rightarrow$~(iii): Now assume that (ii) holds. Clearly (iii)(a) holds.
To prove (iii)(b), let $\theta_1, \theta_2\in \Con(\A)$ and define
$\B := (\A/\theta_1) \times (\A/\theta_2)$. Since $\A/(\theta_1\cap \theta_2)$ embeds into~$\B$,
there must be an embedding $\alpha \colon \B_0 \to \A/(\theta_1\cap \theta_2)$.
By (ii), there exists a homomorphism $\beta \colon \B \to \B_0$.
Thus $\alpha\circ \beta \colon (\A/\theta_1) \times (\A/\theta_2)\to \A/(\theta_1\cap \theta_2)$,
whence condition (iii)(b)~holds.

(iii)~$\Rightarrow$~(i): Finally, assume that (iii) holds. We must show that $\psi$ is surjective.
Let $\B\in \Var(\A)_{\mathrm{fin}}$. Then $\B$ embeds into a finite product $\prod_{i\in I} \C_i$,
where each $\C_i$ is a finite subdirectly irreducible algebra in $\Var(\A)$.

Let $i\in I$. Then $\C_i \to (\C_i)_0$, by~(iii)(a), and therefore $\C_i \equiv (\C_i)_0$.
As $\A$~is the zero-generated free algebra in $\Var(\A)$, we have $(\C_i)_0 \cong \A/\theta_i$
and so $\C_i \equiv \A/\theta_i$, for some $\theta_i \in \Con(\A)$.
It follows that $\prod_{i\in I} \C_i \equiv \prod_{i\in I} \A/{\theta_i}$.

Define $\theta := \bigcap_{i\in I} \theta_i$ in $\Con(\A)$. Since $\A/\theta \to \prod_{i\in I} \A/\theta_i$
always holds, it follows by induction from~(iii)(b) that $\prod_{i\in I} \A/\theta_i \equiv \A/\theta$.
We now have $\prod_{i\in I} \C_i \equiv \A/\theta$, and so $\B \equiv \A/\theta$, using~($\natural$).
Hence the map~$\psi$ is surjective, as required.
\end{proof}

Given an algebra $\A$, let $\A^+$ denote the algebra obtained from~$\A$ by
naming each element via a nullary operation. Since $\Con(\A^+) = \Con(\A)$,
we obtain the following result by applying Lemma~\ref{lem: new} (ii)~$\Rightarrow$~(i)
to the algebra~$\A^+$.

\begin{corollary}\label{cor: retract}
Let $\A$ be a finite algebra, and assume that every finite algebra in~$\Var(\A)$ has a retraction onto
each of its subalgebras. Then $\Lb_{\A^+} \cong \Con(\A)$.
\end{corollary}

We can now easily show that $\Lhom$ contains all finite partition lattices and all finite subspace lattices.
Note that every finite lattice embeds into a finite partition lattice (Pudl\'ak and T\r{u}ma~\cite{PuT}).

\begin{example}\label{ex: retracts}\quad
\begin{enumerate}[label={\upshape(\roman*)}]
\item
For every non-empty finite set~$A$, the lattice $\Equiv(A)$ of all equivalence relations on~$A$ belongs to~$\Lhom$.
\item
For every finite vector space $\V$, the lattice $\Sub(\V)$ of all subspaces of\/~$\V$ belongs to~$\Lhom$.
\end{enumerate}
\end{example}

\begin{proof}
Both parts follow from Corollary~\ref{cor: retract}.
For~(i), note that every non-empty subset of a set is a retract.
For~(ii), note that every subspace of a vector space is a retract
and that congruences correspond to subspaces.
\end{proof}

We can use this example to say something about the first-order theory of~$\Lhom$.

\begin{lemma}
The only universal first-order sentences true in the class~$\Lhom$ are those true in all lattices.
\end{lemma}

\begin{proof}
Consider a universal first-order sentence $\sigma = \forall x_1\dots \forall x_n\ \Phi(x_1,\dots,x_n)$
that fails in a lattice~$\Lb$. We want to show that $\sigma$ fails in a finite lattice.
By Pudl\'ak and T\r{u}ma~\cite{PuT}, it will then follow that $\sigma$~fails in a
finite partition lattice, whence $\sigma$~fails in~$\Lhom$ by Example~\ref{ex: retracts}(i).

The proof of Dean~\cite[Theorem~1]{D56} for equations can be extended to universal sentences.
Choose witnesses $a_1,\dots,a_n \in L$ of the failure of~$\sigma$ in~$\Lb$. Let $A$ be the
subset of~$L$ consisting of the corresponding evaluations of all subterms of terms appearing
in~$\Phi$, and let $\A = \langle A; \le\rangle$ inherit the order from~$\Lb$.
Then $\A$ is a finite ordered set, and the Dedekind--MacNeille completion $\DM(\A)$
is a finite lattice that preserves all existing joins and meets from~$\A$.
Hence $\sigma$ fails in~$\DM(\A)$.
\end{proof}

\subsection*{Intervals in subgroup lattices}

P\'alfy and Pudl\'ak~\cite{PP80} have proved the equivalence of the following two statements:
\begin{itemize}
\item Every finite lattice arises as the congruence lattice of a finite algebra.
\item Every finite lattice arises as an interval in the subgroup lattice of a finite group.
\end{itemize}
So, given the Finite Lattice Representation Problem, it is also natural to investigate
which intervals in subgroup lattices belong to~$\Lhom$.

The next example was provided by Keith Kearnes, and shows that $\Lhom$ contains
every lattice that can be obtained by adding a new top element to an interval
in the subgroup lattice of a finite group.

\begin{example}\label{ex: subgroups}
Let $\G$ be a finite group.
\begin{enumerate}[label={\upshape(\roman*)}]
\item
Let $\A = \langle A ; \{\,\lambda_g \mid g\in G\,\} \rangle$ be a finite $\G$-set
regarded as a unary algebra, and assume that $\A$ has a trivial subalgebra.
Then $\Lb_{\A^+} \cong \Con(\A)$.
\item
Now let\/ $\Hb$ be a subgroup of\/~$\G$, and define $\A$ to be the $\G$-set with universe
$A = \{\, aH \mid a\in G\,\} \dotcup \{\infty\}$ such that $\lambda_g(aH) = gaH$
and $\lambda_g(\infty) = \infty$, for all $g,a\in G$. Then $\Lb_{\A^+} \cong [\Hb, \G] \oplus \mathbf 1$.
\end{enumerate}
\end{example}

\begin{proof}
(i): We will apply Lemma~\ref{lem: new} (ii)~$\Rightarrow$~(i) to the algebra~$\A^+$.
By assumption, there exists $a\in A$ such that $\{a\}$ is a subuniverse of~$\A$.
Let $\B$ be a finite algebra in $\Var(\A^+)$. Since $\B$ is a model of the equational theory of~$\A^+$,
both of the subsets $B\setminus B_0$ and $\{a^\B\}$ of~$B$ are closed under each~$\lambda_g$.
Hence, we may define a homomorphism $\phi\colon\B\to\B_0$ by $\phi(b) = b$, for all $b\in B_0$,
and $\phi(b) = a^\B$, for all $b\in B\setminus B_0$. It follows from Lemma~\ref{lem: new}
that $\Lb_{\A^+} \cong \Con(\A)$.

(ii): By part~(i), it suffices to show that $\Con(\A) \cong [\Hb, \G] \oplus \mathbf 1$.
Note that $\infty/{\theta} = \{\infty\}$, for all $\theta \in \Con(\A)\setminus \{1_\A\}$,
as the action of~$\G$ on $A\setminus \{\infty\}$ is transitive. There are mutually inverse
order-isomorphisms between $\Con(\A)\setminus \{1_\A\}$ and $[\Hb,\G]$, given by
\begin{align*}
\theta \mapsto \K_\theta, \quad &\text{where } K_\theta :=\{\,a\in G \mid aH \equiv H \pmods{\theta}\,\}, \text{ and}\\
\K \mapsto \theta_\K, \quad &\text{where } aH \equiv bH \pmods{\theta_\K} \iff a^{-1}b\in K. \qedhere
\end{align*}
\end{proof}

\subsection*{Five-element lattices}

The lattice $\M_3$ belongs to~$\Lhom$ by Example~\ref{ex: retracts},
since it can be represented as $\Equiv(\{0, 1, 2\})$ or as $\Sub(\mathbb Z_2^2)$.

With the exception of the pentagon lattice~$\N_5$, we now know that all lattices of size up to~5
belong to~$\Lhom$. Our aim for the remainder of this section is to apply Lemma~\ref{lem: new}
to help us find an example of a finite algebra $\A$ such that $\Lb_\A \cong \N_5$.
The algebra $\A$ will be a four-element distributive bisemilattice with all elements named via nullary operations.

An algebra $\A =\langle A; \wedge, \sqcap\rangle$ is said to be a \emph{distributive bisemilattice}
(alternatively, a distributive quasilattice) if both $\wedge$ and $\sqcap$ are semilattice
operations on~$A$ and each distributes over the other; see~\cite{Pl,Ksi,GR}.
Kalman~\cite{Ksi} has proved that, up to isomorphism,
the only subdirectly irreducible distributive bisemilattices
are the algebra~$\D$ and its subalgebras~$\Sb$ and~$\Lb$, shown in Figure~\ref{fig: sld}.
(Note that we depict both $\wedge$ and $\sqcap$ as meet operations.) The algebra~$\Sb$
is term equivalent to the two-element semilattice, and the algebra~$\Lb$ is the
two-element lattice.

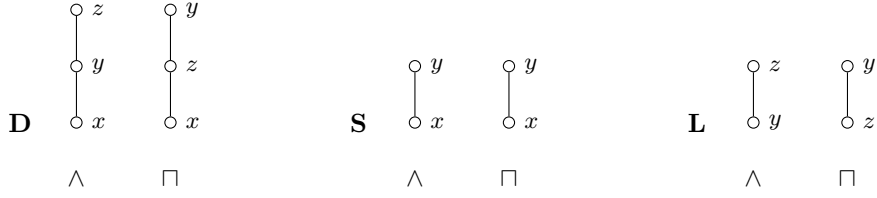
\begin{figure}[tb]
\begin{tikzpicture}[node distance=0.75cm]
\begin{scope}
\begin{scope}
  \node[element,label=right: {\small $x$}] (i) {};
  \node[element,above of=i,label=right: {\small $y$}] (0) {};
  \node[element,above of=0,label=right: {\small $z$}] (1) {};
  \draw (1) -- (0) -- (i);
  \node[below of=i] (m) {$\wedge$};
  \node[left of=i] (l) {$\D$};
\end{scope}
\begin{scope}[xshift=1.25cm]
  \node[element,label=right: {\small $x$}] (i) {};
  \node[element,above of=i,label=right: {\small $z$}] (1) {};
  \node[element,above of=1,label=right: {\small $y$}] (0) {};
  \draw (0) -- (1) -- (i);
  \node[below of=i] (m) {$\sqcap$};
\end{scope}
\end{scope}
\begin{scope}[xshift=4.5cm]
\begin{scope}
  \node[element,label=right: {\small $x$}] (i) {};
  \node[element,above of=i,label=right: {\small $y$}] (0) {};
  \draw (0) -- (i);
  \node[below of=i] (m) {$\wedge$};
  \node[left of=i] (l) {$\Sb$};
\end{scope}
\begin{scope}[xshift=1.25cm]
  \node[element,label=right: {\small $x$}] (i) {};
  \node[element,above of=i,label=right: {\small $y$}] (0) {};
  \draw (0) -- (i);
  \node[below of=i] (m) {$\sqcap$};
\end{scope}
\end{scope}
\begin{scope}[xshift=9cm]
\begin{scope}
  \node[element,label=right: {\small $y$}] (0) {};
  \node[element,above of=0,label=right: {\small $z$}] (1) {};
  \draw (1) -- (0);
  \node[below of=0] (m) {$\wedge$};
  \node[left of=0] (l) {$\Lb$};
\end{scope}
\begin{scope}[xshift=1.25cm]
  \node[element,label=right: {\small $z$}] (1) {};
  \node[element,above of=1,label=right: {\small $y$}] (0) {};
  \draw (0) -- (1);
  \node[below of=1] (m) {$\sqcap$};
\end{scope}
\end{scope}
\end{tikzpicture}
\caption{The subdirectly irreducible distributive bisemilattices.}\label{fig: sld}
\end{figure}

\begin{example}\label{ex: pentagon}
Let $\A = \langle \{0,a,b,1\}; \wedge, \sqcap, 0, a, b, 1\rangle$ be the distributive
bisemilattice $\Sb \times \Lb$, labelled as in Figure~\ref{fig: sxl},
with all four elements named via nullary operations. Then\/ $\Lb_\A \cong \mathbf N_5$.

\begin{figure}[tb]
\begin{tikzpicture}[node distance=1cm]
\begin{scope}
\node[element,label=below: {\small $0 = (x,y)$}] (0) {};
\node[element,above left of=0,label=left: {\small $a = (y,y)$}] (a) {};
\node[element,above right of=0,label=right: {\small $b = (x,z)$}] (b) {};
\node[element,above right of=a,label=above: {\small $1 = (y,z)$}] (1) {};
\draw (1) -- (a) -- (0);
\draw (1) -- (b) -- (0);
\node[below of=0] (m) {$\wedge$};
\end{scope}
\begin{scope}[xshift=6cm]
\node[element,label=below: {\small $b = (x,z)$}] (b) {};
\node[element,above left of=b,label=left: {\small $0 = (x,y)$}] (0) {};
\node[element,above right of=b,label=right: {\small $1 = (y,z)$}] (1) {};
\node[element,above right of=0,label=above: {\small $a = (y,y)$}] (a) {};
\draw (1) -- (a) -- (0);
\draw (1) -- (b) -- (0);
\node[below of=b] (m) {$\sqcap$};
\end{scope}
\end{tikzpicture}
\caption{The distributive bisemilattice $\Sb \times \Lb$.}\label{fig: sxl}
\end{figure}
\end{example}

\begin{proof}
We will apply Lemma~\ref{lem: new}~(iii)~$\Rightarrow$~(i).
The congruences $\alpha$, $\beta$ and $\gamma$ on~$\A$ are shown in Figure~\ref{fig: N5Con}.
Note that the underlying bisemilattices of $\A/\alpha$, $\A/\beta$ and $\A/\gamma$
are isomorphic to~$\Sb$, $\Lb$ and~$\D$, respectively; see Figure~\ref{fig: si}.
We shall see that these are the only subdirectly irreducible algebras in $\Var(\A)$.

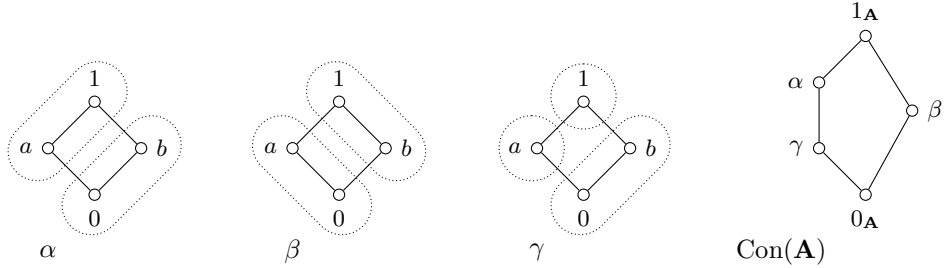
\begin{figure}[tb]
\begin{tikzpicture}[node distance=0.875cm]
\begin{scope}
\node[element,label=below: {\small 0}] (0) {};
\node[element,above left of=0,label=left: {\small $a$}] (a) {};
\node[element,above right of=0,label=right: {\small $b$}] (b) {};
\node[element,above right of=a,label=above: {\small 1}] (1) {};
\draw (1) -- (a) -- (0);
\draw (1) -- (b) -- (0);
\node[below of=a,yshift=-0.5cm] (c) {$\alpha$};
\begin{pgfonlayer}{background}
  \node[potato,fit=(a)(1),rotate=45,inner xsep=5pt,inner ysep=0pt,yshift=2pt] {};
  \node[potato,fit=(0)(b),rotate=45,inner xsep=5pt,inner ysep=0pt,yshift=-2pt] {};
\end{pgfonlayer}
\end{scope}
\begin{scope}[xshift=3.25cm]
\node[element,label=below: {\small 0}] (0) {};
\node[element,above left of=0,label=left: {\small $a$}] (a) {};
\node[element,above right of=0,label=right: {\small $b$}] (b) {};
\node[element,above right of=a,label=above: {\small 1}] (1) {};
\draw (1) -- (a) -- (0);
\draw (1) -- (b) -- (0);
\node[below of=a,yshift=-0.5cm] (c) {$\beta$};
\begin{pgfonlayer}{background}
  \node[potato,fit=(0)(a),rotate=-45,inner xsep=5pt,inner ysep=0pt,yshift=-2pt] {};
  \node[potato,fit=(b)(1),rotate=-45,inner xsep=5pt,inner ysep=0pt,yshift=2pt] {};
\end{pgfonlayer}
\end{scope}
\begin{scope}[xshift=6.5cm]
\node[element,label=below: {\small 0}] (0) {};
\node[element,above left of=0,label=left: {\small $a$}] (a) {};
\node[element,above right of=0,label=right: {\small $b$}] (b) {};
\node[element,above right of=a,label=above: {\small 1}] (1) {};
\draw (1) -- (a) -- (0);
\draw (1) -- (b) -- (0);
\node[below of=a,yshift=-0.5cm] (c) {$\gamma$};
\begin{pgfonlayer}{background}
  \node[potato,fit=(a),inner xsep=5pt,inner ysep=10pt,xshift=-2pt] {};
  \node[potato,fit=(1),inner xsep=5pt,inner ysep=10pt,yshift=2pt] {};
  \node[potato,fit=(0)(b),rotate=45,inner xsep=5pt,inner ysep=0pt,yshift=-2pt] {};
\end{pgfonlayer}
\end{scope}
\begin{scope}[xshift=10.25cm]
\node[element,label=below: {\small $0_\A$}] (0) {};
\node[element,above left of=0,label=left: {\small $\gamma$}] (c) {};
\node[element,above right of=0,label=right: {\small $\beta$},yshift=0.5cm] (b) {};
\node[element,above of=c,label=left: {\small $\alpha$}] (a) {};
\node[element,above right of=a,label=above: {\small $1_\A$}] (1) {};
\draw (1) -- (a) -- (c) -- (0);
\draw (1) -- (b) -- (0);
\node[below of=c,xshift=-0.5cm,yshift=-0.5cm] (c) {$\Con(\A)$};
\end{scope}
\end{tikzpicture}
\caption{The congruences on $\A$ (shown relative to $\wedge$).}\label{fig: N5Con}
\end{figure}

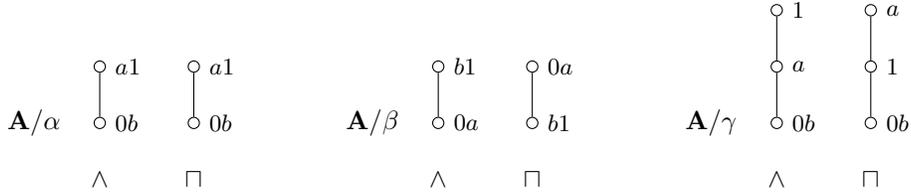
\begin{figure}[tb]
\begin{tikzpicture}[node distance=0.75cm]
\begin{scope}
\begin{scope}
  \node[element,label=right: {\small $0b$}] (i) {};
  \node[element,above of=i,label=right: {\small $a1$}] (0) {};
  \draw (0) -- (i);
  \node[below of=i] (m) {$\wedge$};
  \node[left of=i,xshift=-0.125cm] (l) {$\A/\alpha$};
\end{scope}
\begin{scope}[xshift=1.25cm]
  \node[element,label=right: {\small $0b$}] (i) {};
  \node[element,above of=i,label=right: {\small $a1$}] (0) {};
  \draw (0) -- (i);
  \node[below of=i] (m) {$\sqcap$};
\end{scope}
\end{scope}
\begin{scope}[xshift=4.5cm]
\begin{scope}
  \node[element,label=right: {\small $0a$}] (0) {};
  \node[element,above of=0,label=right: {\small $b1$}] (1) {};
  \draw (1) -- (0);
  \node[below of=0] (m) {$\wedge$};
  \node[left of=0,xshift=-0.125cm] (l) {$\A/\beta$};
\end{scope}
\begin{scope}[xshift=1.25cm]
  \node[element,label=right: {\small $b1$}] (1) {};
  \node[element,above of=1,label=right: {\small $0a$}] (0) {};
  \draw (0) -- (1);
  \node[below of=1] (m) {$\sqcap$};
\end{scope}
\end{scope}
\begin{scope}[xshift=9cm]
\begin{scope}
  \node[element,label=right: {\small $0b$}] (i) {};
  \node[element,above of=i,label=right: {\small $a$}] (0) {};
  \node[element,above of=0,label=right: {\small 1}] (1) {};
  \draw (1) -- (0) -- (i);
  \node[below of=i] (m) {$\wedge$};
  \node[left of=i,xshift=-0.125cm] (l) {$\A/\gamma$};
\end{scope}
\begin{scope}[xshift=1.25cm]
  \node[element,label=right: {\small $0b$}] (i) {};
  \node[element,above of=i,label=right: {\small 1}] (1) {};
  \node[element,above of=1,label=right: {\small $a$}] (0) {};
  \draw (0) -- (1) -- (i);
  \node[below of=i] (m) {$\sqcap$};
\end{scope}
\end{scope}
\end{tikzpicture}
\caption{The subdirectly irreducible algebras in $\Var(\A)$.}\label{fig: si}
\end{figure}

As nullary operations play no role in determining the congruences on an algebra,
the underlying bisemilattice of a subdirectly irreducible algebra in $\Var(\A)$
must be a subdirectly irreducible distributive bisemilattice and hence
must be isomorphic to~$\Sb$, $\Lb$ or~$\D$, by Kalman's result~\cite{Ksi}.
But the values of the four nullary operations on an algebra in $\Var(\A)$ are determined by
the two semilattice operations:
\begin{itemize}
\item $1$ and $0$ are the top and bottom for~$\wedge$, respectively;
\item $a$ and $b$ are the top and bottom for~$\sqcap$, respectively.
\end{itemize}
Hence, the values of the nullary operations $0$, $a$, $b$, $1$ shown in Figure~\ref{fig: si}
are the only possible labellings of the bisemilattices $\Sb$, $\Lb$ and $\D$ that produce an
algebra in $\Var(\A)$. So these are the only subdirectly irreducible algebras
in~$\Var(\A)$. Each of these algebras is zero-generated, so it follows
that condition~\ref{lem: new}(iii)(a) holds.

Since $\gamma \subseteq \alpha$ and $\alpha \cap \beta = 0_\A$, we can establish
condition~\ref{lem: new}(iii)(b) by showing $(\A/\alpha) \times (\A/\beta) \to \A$.
But $\alpha \cap \beta = 0_\A$ and $\alpha \cdot \beta = 1_\A$ imply
$\A \cong (\A/\alpha) \times (\A/\beta)$.
Hence, by Lemma~\ref{lem: new}, we have $\Lb_\A\cong \Con(\A) \cong \mathbf N_5$.
\end{proof}

We have established that all lattices of size up to~5 arise as the homomorphism lattice
induced by a finite algebra.

\subsection*{Acknowledgement}
We are indebted to Keith Kearnes for suggestions that led to improvements in
Section~\ref{sec: extra}. In particular, Example~\ref{ex: subgroups} is due to Keith.


\end{document}